\documentclass[11pt]{article}

\date{}

\usepackage{amsmath, amssymb,amsthm,cite}
\usepackage[utf8x]{inputenc}
\usepackage{graphicx}
\usepackage{bbm}
\usepackage{mathtools}
\usepackage{color}
\usepackage{tcolorbox}
\usepackage{caption}
\usepackage{subcaption}
\usepackage{booktabs}       

\DeclareMathOperator*{\nn}{\nonumber}

\DeclareMathOperator*{\cost}{cost}

\allowdisplaybreaks
\newtheorem{lemma}{Lemma}
\newtheorem{theorem}{Theorem}

\newtheorem{remark}{Remark}
\newtheorem{problem}{Problem}

\DeclareMathOperator{\vx}{\mathbf{x}}
\DeclareMathOperator{\vw}{\mathbf{w}}
\DeclareMathOperator{\vu}{\mathbf{u}}

\DeclareMathOperator{\vs}{\mathbf{s}}

\newcommand\caus[1]{(zI-{#1})^{-1}}
\newcommand\anti[1]{(z^{-1}I-{#1})^{-1}}

\DeclarePairedDelimiterX{\norm}[1]{\lVert}{\rVert}{#1}

\begin{document}
\title{Optimal Competitive-Ratio Control}
\author{Oron Sabag, Sahin Lale, and Babak Hassibi\thanks{The authors are with the Department of Electrical Engineering, California Institute of Technology, Pasadena, CA 91125, USA (e-mails:\{oron,alale,hassibi\}@caltech.edu). }}
\pdfoutput=1

\maketitle
\begin{abstract}
Inspired by competitive policy designs approaches in online learning, new control paradigms such as competitive-ratio and regret-optimal control have been recently proposed as alternatives to the classical $\mathcal{H}_2$ and $\mathcal{H}_\infty$ approaches. These competitive metrics compare the control cost of the designed controller against the cost of a clairvoyant controller, which has access to past, present, and future disturbances in terms of ratio and difference, respectively. While prior work provided the optimal solution for the regret-optimal control problem, in competitive-ratio control, the solution is only provided for the sub-optimal problem. In this work, we derive the optimal solution to the competitive-ratio control problem. We show that the optimal competitive ratio formula can be computed as the maximal eigenvalue of a simple matrix, and provide a state-space controller that achieves the optimal competitive ratio. We conduct an extensive numerical study to verify this analytical solution, and demonstrate that the optimal competitive-ratio controller outperforms other controllers on several large scale practical systems. The key techniques that underpin our explicit solution is a reduction of the control problem to a Nehari problem, along with a novel factorization of the clairvoyant controller's cost. We reveal an interesting relation between the explicit solutions that now exist for both competitive control paradigms by formulating a regret-optimal control framework with weight functions that can also be utilized for practical purposes.

\end{abstract}

\section{Introduction}

Given a dynamical system that is excited by some disturbance process and can be influenced by a control input, the main approach in optimal control is to choose the control signals in such a way to minimize a given regulating cost. For this purpose, there are two classical paradigms adopted in the literature: $\mathcal{H}_2$ and $\mathcal{H}_{\infty}$ control. In $\mathcal{H}_2$ control, the expected control cost is minimized under the assumption of stochastic disturbances, whereas in $\mathcal{H}_{\infty}$ control the worst-case cost is minimized for all possible bounded (energy) disturbances under no distributional assumption. For both methods, degradation of performance can occur if the assumptions on the disturbances are not met. For instance, in $\mathcal{H}_2$ control, this can happen when the statistics of the disturbances are not what the $\mathcal{H}_2$ controller was designed for \cite{doyle1978guaranteed}. On the other hand, in $\mathcal{H}_\infty$ control, the designed controller may be too conservative since it is designed with respect to the worst-case scenario (disturbance) which may not be the typical scenario in practical systems~\cite{GloverDoyle_book}.

Motivated by competitive design approaches in online learning~\cite{shalev2012online}, two control paradigms have been recently proposed as alternatives to the classical paradigms, namely regret-optimal control~\cite{sabagFIACC} and competitive-ratio control~\cite{goel2021competitive}. These criteria aim to minimize the difference or the ratio between the costs of a causal controller (to be designed) and the cost of the optimal clairvoyant controller that has access to entire disturbances sequence. The corresponding competitive metric, i.e., the difference or the ratio between the costs is a function of the disturbance sequence, and the designed controller aims to minimize the maximal competitive criterion over all disturbance sequences.

In the full-information control problem studied in this paper, the cost of the optimal clairvoyant (non-causal) controller yields a universal lower bound on the cost of any other controller (linear or non-linear, causal or non-causal). In the competitive approach, we attempt to follow the performance of the optimal non-causal controller as closely as possible. In the regret-optimal case, we choose a control strategy to minimize the worst-case (over bounded energy disturbances) difference between the causal and optimal non-causal costs. In the competitive-ratio case, we minimize the corresponding worst-case ratio of these two costs. These strategies result in novel controller behavior, e.g., the regret-optimal controller interpolates between $\mathcal{H}_2$ and $\mathcal{H}_{\infty}$ to achieve the best of both worlds~\cite{sabagFIACC}, and provide superior performance in various control tasks with various disturbance characteristics~\cite{sabagFIACC,SabagFilteringAISTATS,goel2021competitive}. \looseness=-1

The regret-optimal control problem has been solved in \cite{sabagFIACC}, meaning that the optimal value of the regret has been explicitly given and a corresponding state-space controller has been provided \cite{sabag2021regret}. In the competitive ratio case, however, only the sub-optimal problem (where the competitive ratio is bounded by a given constant) has been solved, meaning that the optimal solution can only be found by performing a bisection over this constant \cite{goel2021competitive}\footnote{This is akin to the situation in standard $\mathcal H_\infty$ control.}.

\textbf{Contributions.} The main result in this work is the optimal solution to the competitive-ratio control problem. An optimal solution is referring to an explicit formula for the optimal (minimal) competitive ratio, which is given in terms of a maximal eigenvalue of a simple matrix, and to an explicit state-space controller that achieves this optimal competitive ratio. To obtain these explicit results, we use an operator-theoretic approach to show that the competitive-ratio control problem can be formulated as a Nehari problem~\cite{nehari1957bounded}. The advantage of this formulation is that the Nehari problem can be solved optimally, and its state-space solution has been recently characterized~\cite{sabag2021regret}. The resulting controller design simply requires a solution to two Riccati, and three (linear) Lyapunov equations in the general case. Given our explicit solution, we show that the optimal competitive-ratio controller simplifies in special cases of interest. For instance, for scalar systems, the optimal competitive-ratio is achieved with the standard $\mathcal{H}_2$ (LQR) state-feedback controller. \looseness=-1

The approach in this paper to reduce the control problem to a Nehari problem is similar to the methodology used to solve the regret-optimal control problem \cite{sabag2021regret}. To underpin the reason that these problems can be solved optimally (rather than the sub-optimal solution to $\mathcal H_\infty$ control), we generalize the regret-optimal control framework to include weight operators on the system state, input, and disturbances. We show that this generalized problem can be formulated as a regret-optimal control problem without weights (but with a modified system parameters) and therefore it has an optimal solution from \cite{sabag2021regret}. Interestingly, we show that competitive-ratio control can be formulated as a regret-optimal control problem with weights. Unlike the weights used in synthesizing $\mathcal{H}_{\infty}$ controllers that depend on mostly domain knowledge, e.g. weighting on the disturbance spectrum, the weights in regret-optimal control can be system dependent through the weighting on the non-causal controller. This generalized regret framework provides a new perspective on control system design in practice.


\textbf{Related Work.} The competitive-ratio and regret are competitive metrics that quantify the performance of online learning algorithms in the context of sequential decision-making for memoryless environments against the optimal policy in  hindsight~\cite{borodin2005online,shalev2012online,zinkevich2003online,gordon1999regret}. This concept has been adapted to linear dynamical systems with stochastic~\cite{abbasi2011regret,cohen19b,lale2020logarithmic, muthirayan2021online, lale2022reinforcement}, or adversarial disturbances~\cite{agarwal2019online,foster2020logarithmic,simchowitz2020making}, as well as nonlinear dynamical systems~\cite{sattar2021identification,boffi2021regret,lale2021model}. These studies focus on control algorithms that compete against a best fixed controller from some parametric class of policies, \textit{i.e.}, policy regret. In the current work, the reference policy is not restricted to a parametric class, but is chosen as the universal optimal clairvoyant controller. Another key distinction is that our solution is explicit and achieves the \emph{optimal competitive ratio}.

Related to the current work is the regret-optimal control framework that also competes against the \emph{globally} optimal clairvoyant controller. This framework aims to produce controllers mimic the performance of the clairvoyant controller by minimizing the regret.  The formulation of the regret-optimal control problem and the regret-optimal controller were introduced in \cite{sabagFIACC,sabag2021regret} for the full-information control problem, and was later extended to the finite-horizon regime in~\cite{goel2021regret}. The controller for the finite-horizon regime has been extended to incorporate state and input safety constraints  in~\cite{martin2022safe,didier2022system} by leveraging the System Level Synthesis framework~\cite{wang2019system}, and to regret bound for the $\mathcal H_\infty$ controller \cite{karapetyan2022regret}. The regret-optimal framework studies the controllers that minimize the difference between costs, while the competitive ratio studied here minimizes the costs ratio. In Section \ref{sec:numerical}, we compare the performance of these two paradigms.

Bounds on the competitive ratio has been the focus of many prior works in control~\cite{goel2019online,shi2020online,yu2020competitive,li2022robustness}. The competitive-ratio control problem has been proposed in \cite{goel2021competitive}, where its sub-optimal solution was derived under certain conditions. In the current paper, we provide the solution for the \emph{optimal competitive-ratio control problem}. This is a surprising result since typical solutions to control problems with adversarial disturbance (i.e., {\tt min-max} problems) can be solved via their corresponding sub-optimal problems. \looseness=-1


\section{Problem Formulation and Preliminaries}\label{sec:setting}
In this section, we present the state-space setting of the linear quadratic regulator (LQR) control problem and revisit some of the properties of the optimal non-causal controller. We then present the competitive-ratio control problem, and compare its objective against the formulations of $\mathcal{H}_{\infty}$ and the regret-optimal controls. Finally, we define several Riccati and Lyapunov equations that are needed for the main results. \looseness=-1
\subsection{Notation}
The operator norm of a matrix $A$ is denoted by $\|A\|$. We use $A^\ast$ and $x^\ast$ to denote the conjugate transpose of a matrix $A$ and a vector $x$, respectively. The largest eigenvalue of a positive semidefinite matrix $A$ is denoted by $\lambda_{\text{max}}(A)$, and the square root of $A$ is denoted by $A^{1/2}$, i.e., $A=A^{1/2}A^{1/2}$. We use $I$ to denote the identity matrix when the dimensions are clear from the context. Caligraphic letters e.g. $\mathcal{X}$ denote doubly-infinite linear operators and bold-face letters denote two-sides sequences, e.g., $\vx=\{x_i\}$. The strictly-causal part of a linear operator $\mathcal A$ (i.e., its strictly lower triangular part) is denoted by $\{\mathcal A\}_+$, and the anti-causal part of $\mathcal A$ (its upper triangular) is denoted by $\{\mathcal A\}_-$. An operator is said to be causal if its anti-causal part is zero.

\subsection{The State-Space Setting}
We study the LQR control problem for time-invariant dynamical systems
\begin{align}\label{eq:state-space}
    x_{t+1}&= A x_t + B_u u_t + B_w w_t
\end{align}
where $x_t \in \mathbb{R}^n$ is the state, $u_t \in \mathbb{R}^p$ is the control signal, and $w_t \in \mathbb{R}^m$ is the disturbance. We focus on the quadratic cost at time $t$, $c_t = x_t^\ast Q x_t + u_t^\ast R u_t$ with $Q,R\succ0$, and the corresponding infinite-horizon cumulative cost is
\begin{align}\label{eq:cum_cost}
    \cost(\vu,\vw) = \sum_{t=-\infty}^\infty c_t.
\end{align}
Without loss of generality, we assume $R = I$ by scaling $B_u$ with $B_uR^{-/2}$. We also assume that the pair $(A,B_u)$ is stabilizable and $B_w$ is full column-rank. A controller is defined as a sequence of strictly causal mappings from the disturbance to the control signal. That is, at time $t$, the control signal $u_t$ is a function of $\{w_i\}_{i=-\infty}^{t-1}$. For the sake of competitive-ratio control, it can be shown that the optimal controller can be described with a linear mapping. We remark that the techniques used in this paper can be directly extended to causal controllers (see \cite{sabag2021regret}). Our main results also show that the optimal controller does not really need an access to the underlying disturbances but only to the system states (Remark \ref{remark:state_sufficient}). We consider disturbance sequences $\vw$ that have bounded energy, i.e., $\vw\in\ell_2$, which is a necessary condition to make the cumulative cost in \eqref{eq:cum_cost} finite.

Before presenting our problem formulation, it is convenient to depart from the state-space setting to the general control problem described in terms of linear operators. In particular, the state-space in \eqref{eq:state-space} can be equivalently represented as
\begin{align}\label{eq:operator_sys}
    \vs &= \mathcal Q^{1/2}\vx = \mathcal F \vu +\mathcal G \vw,
\end{align}
where $\vs$ is the state $\vx$ scaled with the block-diagonal operator $\mathcal Q^{1/2}$ (it has $Q^{1/2}$ on its main diagonal), and $\mathcal F, \mathcal G$ are time-invariant, causal (i.e., lower triangular), Toeplitz operators with Markov parameters $F_i = Q^{1/2}A^{i-1}B_u$ and $G_i = Q^{1/2}A^{i-1}B_w$ for $i>0$, respectively. A linear controller is a mapping from the space of disturbance sequences $\vw$ to the space of control signals sequences $\vu$ denoted by $\vu=\mathcal K \vw$. For a fixed $\mathcal{K}$ and $\vw$, the cost in \eqref{eq:cum_cost} can be written as
\begin{align}\label{eq:cost_OP}
 \cost(\mathcal K,\vw) &= \|\vs\|_2^2 + \|\vu\|_2^2 = \vw^\ast T_{\mathcal K}^\ast T_{\mathcal K} \vw,
\end{align}
 where $T_{\mathcal K}$ is the cost operator
 \begin{align}
    \left[ \begin{array}{c} \vs \\ \vu \end{array} \right] = \underbrace{\left[ \begin{array}{c} \mathcal F \mathcal K+\mathcal G \\ \mathcal K \end{array} \right]}_{\triangleq T_{\mathcal K}} \vw.
    \label{transfer_operator}
\end{align}

For the cost operator $T_{\mathcal K}$, the following fundamental identity is known (e.g. \cite{sabagFIACC,hassibi1999indefinite}),
\begin{align}\label{eq:identity}
    T_{\mathcal K}^\ast T_{\mathcal K} & =  ({\mathcal K} - {\mathcal K}_0)^* (I+\mathcal F^*\mathcal F) (\mathcal K - \mathcal{K}_0) + T_{\mathcal K_0}^*T_{\mathcal K_0}
\end{align}
with $\mathcal K_0 = -(I+\mathcal F^*\mathcal F)^{-1} \mathcal F^\ast \mathcal G$, and
$$T_{\mathcal K_0}^*T_{\mathcal K_0} = \mathcal G^*(I + \mathcal F\mathcal F^*)^{-1}\mathcal G.$$
The identity \eqref{eq:identity} implies that $\mathcal K_0$ is the optimal (linear) non-causal controller. Moreover, it implies that $T_\mathcal K^\ast T_\mathcal K\succeq T_{\mathcal K_0}^*T_{\mathcal K_0}$, so that $\mathcal K_0$ outperforms (in terms of LQR cost) any linear controller $\mathcal K$ {\em for any} $\vw$. Indeed, it can also be shown that $\mathcal K_0$ is optimal among non-linear  controllers as well \cite{sabag2021regret}.

\subsection{Control objectives}
\textbf{$\mathcal H_\infty$ control:} In robust control, the objective is to solve the optimization problem
\begin{align}\label{eq:Hinf}
       \underset{{\text{Strictly Causal}\ \mathcal K}}{\inf} \sup_{\|\vw\|\le 1} \cost(\mathcal K,\vw)&= \underset{{\text{S. Causal} \ \mathcal K}}{\inf}
\|T_\mathcal K \|^2.
\end{align}
This problem is challenging to solve directly and, therefore, its solution is given in terms of the sub-optimal solution \cite{GloverDoyle_book}: for a fixed $\gamma$, find a causal controller $\mathcal K$ (if exists) such that
\begin{align}\label{eq:Hinf_sub}
&\|T_\mathcal K \|^2\le \gamma^2\\
&\iff (\mathcal K - \mathcal K_0)^* (I+\mathcal F^*\mathcal F) (\mathcal K - \mathcal K_0) + T_{\mathcal K_0}^*T_{\mathcal K_0} \preceq \gamma^2 I. \nn
\end{align}
Thus, the solution to the $\mathcal H_\infty$ control problem is given as an iterative solution to the sub-optimal problem above. The second line in \eqref{eq:Hinf_sub} follows from the identity in \eqref{eq:identity} and reveals that the performance of the $\mathcal H_\infty$ controller has an inherent lower bound $T_{\mathcal K_0}^*T_{\mathcal K_0}$, which is independent of the designed controller $\mathcal K$, and follows from the cost of $\mathcal K_0$.

\textbf{Regret-Optimal Control:} In regret-optimal control, the objective is to minimize the cost difference of the designed causal controller and the non-causal controller. This is a competitive criterion and the main idea is to obtain a causal controller whose performance tracks the behaviour of the optimal non-causal controller across all disturbances. The regret-optimal control problem is given by
\begin{align}\label{prob:regret}
    \text{Regret} &=\inf_{\text{S. Causal} \ \mathcal K} \sup_{\|\vw\|\le 1} \left( \cost(\mathcal K,\vw) - \cost(\mathcal K_0 ,\vw)\right) \nn\\
    &=\inf_{\text{S. Causal} \ \mathcal K} \| T_{\mathcal K}^\ast T_{\mathcal K}  -T_{\mathcal K_0}^*T_{\mathcal K_0} \|.
\end{align}
In terms of solution to the regret-optimal control problem it may be surprising that the problem can be solved optimally in terms of a simple formula for the regret and a controller that achieves the optimal regret \cite{sabagFIACC,sabag2021regret}.

\textbf{Competitive-Ratio Control:} The competitive-ratio control problem also has a competitiveness property with respect to the non-causal control, yet, through a ratio. In particular, one aims to find the causal controller which minimizes the ratio between the cost of the causal and the non-causal controller. Formally, it is defined as
\begin{align}\label{eq:prof_CRdef}
    \text{Comp-Ratio}&= \inf_{\text{S. Causal} \ \mathcal K} \sup_{\vw} \frac{\cost(\mathcal K,\vw) }{\cost(\mathcal K_0 ,\vw) }.
\end{align}
For the competitive ratio to be well-defined, we assume that the operator $\mathcal G$ whose corresponding transfer function is $G(e^{j\omega})$ has a full column rank for all $\omega$.

From the defined control objectives, the attempted behavior of each controller becomes transparent. In $\mathcal{H}_{\infty}$, one aims to directly minimize the cost incurred by the controller. On the other hand, in regret-optimal and competitive-ratio control, one aims to minimize the difference and ratio in costs over all disturbances, respectively. It is important to note that that costs that are incurred by the causal and the non-causal controller are evaluated at the same disturbance sequence $\vw$. Thus, $\vw$ plays an active role at comparing the costs either via a difference (regret) or the competitive ratio. At the intuitive level, by taking a supremum over $\vw$, for those disturbances that even the best non-causal controller has a large control cost, the regret-optimal controller may have a higher cost. However, if for certain disturbances a lower cost is attainable for the non-causal controller, then the designed controller should have a lower cost as well. This competitive behavior will be illustrated numerically in Section \ref{sec:numerical}.

\subsection{Riccati and Lyapunov equations}
To present our main results, we utilize the stabilizing solution to the standard LQR Riccati equation
\begin{align}\label{eq:Riccati}
  P &= A^\ast P A + Q - K_{\text{lqr}}^\ast (I + B_u^\ast P B_u )K_{\text{lqr}}
\end{align}
with the state-feedback control law $K_{\text{lqr}} \!\triangleq\! (I \!+\! B_u^{*}PB_u)^{-1}B_u^{*}PA$.
Further, we define the Riccati equations \looseness=-1
\begin{align}\label{eq:Ricc_TM}
      T &= ATA^\ast + B_u B_u^\ast - ATQ^{1/2}R_T^{-1} Q^{1/2} TA^\ast\\
      M&= A_T^\ast M A_T + A_T^\ast Q^{1/2}R_T^{-1}Q^{1/2} A_T - K_M^\ast R_M^{-1}K_M,\nn
\end{align}
with $R_T = I + Q^{1/2}TQ^{1/2}$, $R_M = B_w^\ast Q^{1/2}R_T^{-1}Q^{1/2} B_w + B_w^\ast M B_w$, and $K_M = R_M^{-1}(B_w^\ast MA_T + B_w^\ast R_T^{-1} Q^{1/2}A_T)$. The corresponding closed-loop systems are given by
\begin{align}\label{eq:closed-loop}
    A_K &= A- B_u K_{\text{lqr}}, \ \ A_T = A - AT(Q^{-1}+T)^{-1}\nn\\
    A_M &= A_T - B_wK_M.
\end{align}
The following remark shows how to find the solution to the Riccati equation of $T$ by simpler means.
\begin{remark}
The first Riccati equation in \eqref{eq:Ricc_TM} (with $T$) is the dual Riccati equation to the LQR Riccati in \eqref{eq:Riccati} \cite[App. E.8]{kailath2000linear}. Thus, using \cite[Th. E.8.1]{kailath2000linear}, we can compute directly
\begin{align}
    T = O(I-PO)^{-1},
\end{align}
where $O$ is the solution to the (linear) Lyapunov equation
\begin{align}
O= A_K^\ast O A_K + B_u(I + B_u^\ast PB_u)^{-1}B_u^\ast.
\end{align}
\end{remark}
\begin{remark}\label{remark:M}
If $B_w$ is square, the stabilizing solution to the Riccati equation in \eqref{eq:Ricc_TM} is simply $M=0$. Using this fact, we derive a simplified solution to the competitive-ratio control problem in Theorem \ref{th:square_SC}.
\end{remark}

We also define $Z_\gamma$ and $\Pi$ as the solutions to the following Lyapunov equations
\begin{align}\label{eq:lyapunov}
    Z_\gamma &= A_K Z_\gamma A_K^\ast + \gamma^{-2}B_u (I + B_u^\ast P B_u)^{-1} B_u^\ast\\
    \Pi &= A_K^\ast \Pi A_K + (P-A_K^\ast U)B_wR_M^{-1}B_w^\ast (P-A_K^\ast U)^\ast.\nn
\end{align}
The solution to the first equation, $Z_\gamma$, is parameterized by the scalar $\gamma$, and we will use this equation either with $\gamma^2 = 1$ which results in the solution $Z_1$, or with $\gamma^2 = \lambda_{\max}(Z_1\Pi)$ which results in a solution that is denoted by $Z_\ast$. Finally, we define the Sylvester equation
\begin{align}\label{lya:U}
    U &= A_K^\ast U A_M + P B_wK_M.
\end{align}

\section{Main results}\label{sec:main}
In this section, we present our main results. The following theorem presents the optimal solution for the competitive-ratio control problem in \eqref{eq:prof_CRdef}.
\begin{theorem}[Optimal Competitive-Ratio Control]\label{th:SC}
For the LQR problem, the optimal competitive ratio in \eqref{eq:prof_CRdef} is
\begin{align}\label{eq:th_optCR}
    \emph{Comp-Ratio} &= 1 + \lambda_{\text{max}}(Z_1{\Pi}),
\end{align}
where $Z_1$ and ${\Pi}$ are given by \eqref{eq:lyapunov}.

An optimal competitive-ratio controller that achieves $\emph{Comp-Ratio}$ in \eqref{eq:th_optCR} is given by
\begin{align}\label{eq:th_cont}
    u_t&= - K_{\text{lqr}}x_t \mspace{-3mu} + \mspace{-3mu}(I + B_u^\ast P B_u)^{-1}
    B_u^\ast \begin{pmatrix}
     U & - \Pi
    \end{pmatrix}
    \begin{pmatrix}
    \xi^1_{t}\\
    \xi^2_{t}
    \end{pmatrix}
\end{align}
with
\begin{align}\label{eq:th_SC_intSTate}
    \begin{pmatrix}
    \xi^1_{t+1}\\
    \xi^2_{t+1}
    \end{pmatrix} &= \begin{pmatrix}
    A_T&0\\
     K_\gamma K_M &F_\gamma
    \end{pmatrix}
    \begin{pmatrix}
    \xi^1_{t}\\
    \xi^2_{t}
    \end{pmatrix}  +\begin{pmatrix}
    B_w\\
    K_\gamma
    \end{pmatrix} w_t,
\end{align}
where $U,\Pi,A_T,K_M,Z_\ast$ are defined in \eqref{eq:Ricc_TM}-\eqref{lya:U}, and the constants $K_\gamma,F_\gamma$ are given by
\begin{align}
        K_\gamma &=  (I - A_K Z_\ast A_K^\ast \Pi )^{-1}A_K Z_\ast (P-A_K^\ast U)B_w\nn\\
    F_\gamma &= A_K - K_\gamma R_M^{-1}B_w^\ast (P- U^\ast A_K).
\end{align}
\end{theorem}
\begin{remark}\label{remark:state_sufficient}
From operational perspective, the controller in Theorem \ref{th:SC} has the advantage that it can be implemented with access to the system states only rather than the disturbances. This fact can be seen from \eqref{eq:th_SC_intSTate} since the evolution of the internal state is as a function of $B_ww_t = x_{t+1}-Ax_t-B_u u_t$.
\end{remark}
The optimal competitive ratio in \eqref{eq:th_optCR} can be computed directly as the maximal eigenvalue value of a finite-dimensional matrix. Furthermore, \eqref{eq:th_cont} provides an explicit controller which is composed of the LQR state-feedback law and a controller whose evolution is given in \eqref{eq:th_SC_intSTate}. This resolves the competitive-ratio control problem that was introduced in~\cite{goel2021competitive}.

To the best of our knowledge, the controller in Theorem~\ref{th:SC} and the regret-optimal controller in \cite{sabag2021regret} are the only explicit solutions for control problems in the robust control paradigm. This is a departure from the sub-optimal solution that exists for the robust $\mathcal H_\infty$ controller. In the next section, we present a key observation behind our optimal solution to the competitive-ratio control problem. This sheds light onto the explicit solution derived in Section \ref{sec:state-space} and draw some underlying connections with the regret-optimal control framework.

The following result is for the special case where $B_w$ is a square matrix. We present this special case as a separated result since the optimal competitive-ratio controller is simplified significantly due to the fact that $M=0$ (see Remark \ref{remark:M}).
\begin{theorem}[Optimal Competitive-Ratio Control: Square $B_w$]\label{th:square_SC}
For the setting in Theorem \ref{th:SC} with a square $B_w$, the optimal competitive-ratio is given by
\begin{align}\label{eq:th_sq_optCR}
    \emph{Comp-Ratio} &= 1 + \lambda_{\text{max}}(Z_1\overline{\Pi}),
\end{align}
where $Z_1$ is given in \eqref{eq:lyapunov}, and $\overline{\Pi}$ solves
\begin{align}
    \overline{\Pi}&= A_k^\ast \overline{\Pi} A_k + (P - A_K^\ast PA_T)(Q^{-1}+T) (P - A_T^\ast P A_K).\nn
\end{align}
Let $\overline{Z}_\ast$ solve the Lyapunov equation in \eqref{eq:lyapunov} with $\gamma^2 \!=\! \lambda_{\text{max}}(Z_1\overline{\Pi})$, then an optimal competitive-ratio controller is given by \looseness=-1
\begin{align*}
    u_t&= - K_{\text{lqr}}x_t + (I + B_u^\ast P B_u)^{-1}
    B_u^\ast \begin{pmatrix}
     P A_T & - \overline{\Pi}
    \end{pmatrix}
    \begin{pmatrix}
    \xi^1_{t}\\
    \xi^2_{t}
    \end{pmatrix},
\end{align*}
where the controller states evolve by the dynamical system
\begin{align}
    \begin{pmatrix}
    \xi^1_{t+1}\\
    \xi^2_{t+1}
    \end{pmatrix}
    &= \begin{pmatrix}
    A_T&0\\
     \overline{K}_\gamma A_T & \overline{F}_\gamma
    \end{pmatrix}
    \begin{pmatrix}
    \xi^1_{t}\\
    \xi^2_{t}
    \end{pmatrix} + \begin{pmatrix}
    I\\
    \overline{K}_\gamma
    \end{pmatrix}B_w w_t.
\end{align}
The constants that appear above are given in \eqref{eq:Riccati}-\eqref{eq:closed-loop} and
\begin{align}\label{eq:th_SC_KF}
    \overline{K}_\gamma &=  (I - A_K \overline{Z}_\ast A_K^\ast \overline{\Pi} )^{-1}A_K \overline{Z}_\ast (P - A_K^\ast PA_T)\\
    \overline{F}_\gamma &= A_K - \overline{K}_\gamma Q^{-/2}R_TQ^{-/2} (P - A_T^\ast PA_K).\nn
\end{align}
\end{theorem}
\begin{remark}
Similar to \cite{goel2021competitive}, we note that in the case of square $B_w$\footnote{If $B_w$ is a square matrix, our assumption that $B_w$ has a full column-rank implies that $B_w$ is also invertible.}, the optimal competitive ratio is independent of $B_w$.
\end{remark}
Note that the structures of the controllers given in Theorems \ref{th:SC} and \ref{th:square_SC} are similar. However, the main simplification of Theorem \ref{th:square_SC} lies in the constants computation. In light of Remarks $1-2$, when $B_w$ is square, we only need to solve the LQR Riccati equation in \eqref{eq:Riccati} and linear matrix equations.

In the scalar case, the controller in Theorem \ref{th:square_SC} simplifies further and leads to the following surprising result.
\begin{theorem}[Scalar systems]\label{th:scalar}
For scalar systems (i.e., $A,B_u,B_w,Q\in\mathbb{R}$), the optimal competitive-ratio in Theorem~\ref{th:SC} simplifies to
\begin{align}
    \emph{Comp-Ratio} &= 1 + \frac{B_u^2 P^2(Q^{-1}+P)}{1+B_u^2P}.
\end{align}
Moreover, the competitive-ratio optimal controller is
\begin{align}
    u_t&= - K_{\text{lqr}} x_t.
\end{align}
In other words, for scalar systems, the $\mathcal{H}_2$ controller with state-feedback law attains the optimal competitive ratio.
\end{theorem}

\section{Main ideas and Extensions}\label{sec:op}
In this section, we present the derivation of the competitive-ratio optimal controller at the operator level. We then reveal an underlying similarity between competitive ratio and regret by studying a regret-optimal control problem with weight functions. We start by describing a fundamental problem that lies at the heart of competitive-ratio control.
\begin{problem}[Nehari Problem \cite{nehari1957bounded}]\label{prob:nehari} Given a strictly anti-causal (strictly upper triangular) operator $U$, find a causal (lower triangular) operator $L$, such that $\|L-U\|$ is minimized.
\end{problem}
The Nehari problem seeks the best causal approximation for a strictly anti-causal operator in the operator norm sense. The problem has been solved, and the minimal norm can be characterized by the Hankel norm of $U$ \cite{nehari1957bounded}. In the case of an operator $U$ that can be described with a state-space, the Hankel norm can be computed explicitly and we can also characterize the state-space representation of the approximation $L$ \cite{sabag2021regret}. We proceed to show that competitive-ratio control can be solved optimally using Problem \ref{prob:nehari}.
\subsection{Optimal Competitive-ratio Control}
The competitive-ratio control problem is defined as
\begin{align}\label{eq:CR_start}
     \text{Comp-Ratio}&= \inf_{\text{S. causal}\ \mathcal K }\sup_{\vw} \frac{\vw^\ast T_{\mathcal K}^\ast T_{\mathcal K} \vw}{\vw^\ast T_{\mathcal K_0}^*T_{\mathcal K_0} \vw}.
\end{align}
We derive the optimal solution to \eqref{eq:CR_start} via the corresponding sub-optimal problem: For a fixed $\gamma\in \mathbb{R}^+$, find a strictly causal controller $\mathcal K$ (if exists) such that
\begin{align}\label{eq:cr_sub0}
    &\sup_{\vw} \frac{\vw^\ast T_{\mathcal K}^\ast T_{\mathcal K} \vw}{\vw^\ast T_{\mathcal K_0}^*T_{\mathcal K_0} \vw}\le \gamma^2.
\end{align}
The minimal value of $\gamma^2$ such that \eqref{eq:cr_sub0} is satisfied is equal to the optimal competitive ratio in \eqref{eq:CR_start}. To quantify this minimal value in \eqref{eq:cr_sub0}, we strive to formulate \eqref{eq:cr_sub0} as a Nehari problem since it implies that the problem can be solved optimally. We start by writing \eqref{eq:cr_sub0} as
\begin{align}\label{eq:cr_sub}
    & \vw^\ast T_{\mathcal K}^\ast T_{\mathcal K} \vw\le \gamma^2 \vw^\ast T_{\mathcal K_0}^*T_{\mathcal K_0} \vw, \ \  \forall \vw.
\end{align}
The following chain of problems are equivalent to \eqref{eq:cr_sub}
\begin{align}
    &\vw^\ast (\mathcal K - \mathcal K_0)^* (I+\mathcal F^*\mathcal F) (\mathcal K - \mathcal K_0) \vw\nn\\
    &\ \ \le (\gamma^2-1) \vw^\ast \mathcal G^*(\mathcal I + \mathcal F\mathcal F^*)^{-1}\mathcal G \vw, \ \ \ \label{eq:cr_eq_1}\\
    & \vw^\ast (\Delta \mathcal K - \Delta \mathcal K_0)^* (\Delta \mathcal K - \Delta \mathcal K_0) \vw\nn\\
    &\ \ \le (\gamma^2-1) \vw^\ast \mathcal M^\ast \mathcal M \vw, \ \ \ \label{eq:cr_eq_4},
\end{align}
where \eqref{eq:cr_eq_1} follows from the fundamental identity in \eqref{eq:identity}, and \eqref{eq:cr_eq_4} follows from the canonical factorizations $\Delta^\ast \Delta = I+\mathcal F^*\mathcal F$ and $\mathcal M^\ast \mathcal M = G^*(\mathcal I + \mathcal F\mathcal F^*)^{-1}\mathcal G$. These factorizations necessitate that $\mathcal M,\Delta$ are causal operators, and their inverses are causal and bounded. \looseness=-1

We now assume that $\mathcal M$ is a bounded operator (a fact that will be shown formally below) which implies that we can change the disturbance variable as $\vw' = \mathcal M \vw$, and conclude that the competitive-ratio control problem in \eqref{eq:cr_sub0} is equivalent to \looseness=-1
\begin{align}
    & \vw^\ast (\Delta \mathcal K \mathcal M^{-1} - \Delta \mathcal K_0\mathcal M^{-1})^* (\Delta \mathcal K \mathcal M^{-1}- \Delta \mathcal K_0 \mathcal M^{-1}) \vw\nn\\
    &\ \ \le (\gamma^2-1) \vw^\ast \vw, \ \forall \vw. \label{eq:cr_eq_5}
\end{align}

The problem in \eqref{eq:cr_eq_5} resembles the Nehari problem in Problem \ref{prob:nehari}. In particular, divide both sides with $\vw^\ast \vw$ and take a supremum over $\vw$ to obtain
\begin{align}\label{eq:der_some}
 \|\Delta \mathcal K \mathcal M^{-1} - \Delta \mathcal K_0\mathcal M^{-1}\|^2\le \gamma^2-1.
\end{align}
As $\Delta$ and $\mathcal M^{-1}$ are causal operators, the product $\Delta \mathcal K \mathcal M^{-1}$ is strictly causal. On the other hand, the operator $\Delta \mathcal K_0\mathcal M^{-1}$ is neither causal or anti-causal since $\mathcal K_0$ is a non-causal mapping. Thus, we decompose this operator into its anti-causal and strictly causal counterparts
\begin{align}\label{eq:some_label}
    \Delta \mathcal K_0\mathcal M^{-1}&= \{\Delta \mathcal K_0\mathcal M^{-1}\}_+ + \{\Delta \mathcal K_0\mathcal M^{-1}\}_-.
\end{align}
By defining the strictly causal operator $\mathcal K' = \Delta \mathcal K \mathcal M^{-1} - \{\Delta \mathcal K_0\mathcal M^{-1}\}_+$, we obtain that \eqref{eq:der_some} is the Nehari problem
\begin{align}\label{eq:op_comp_final_nehari}
    \inf_{\text{S. causal}\ \mathcal K' } \| \mathcal K' - \{\Delta \mathcal K_0\mathcal M^{-1}\}_-\|^2,
\end{align}
and if $\mathcal K'$ is the solution to \eqref{eq:op_comp_final_nehari}, then an optimal competitive-ratio controller is obtained as
\begin{align}\label{eq:op_optcon}
 \mathcal K = \Delta^{-1}(\mathcal K' + \{\Delta \mathcal K_0\mathcal M^{-1}\}_+)\mathcal M.
\end{align}
Note that the constant on the right hand side of \eqref{eq:cr_eq_5} is $\gamma^2-1$ and, therefore, the optimal competitive ratio in \eqref{eq:cr_sub0} is equal to the minimal value of the Nehari problem in \eqref{eq:op_comp_final_nehari} plus $1$. This is compatible with operational definition of competitive ratio whose trivial lower bound is~$1$.

\subsection{Regret-Optimal Control with Weights}
In this section, we present a generalization of the regret-optimal control problem to have weight functions. This generalization is useful for practical scenarios where domain-knowledge aim at attenuating particular disturbances, control signals or systems' states. From analytical perspective, we show that the problem can be solved optimally and that the competitive-ratio control problem is an instance of regret-optimal control with particular weight functions.

Recall that the regret-optimal control problem in \eqref{prob:regret} is
\begin{align}
    \inf_{\mathcal K} \|T_{\mathcal K}^*T_{\mathcal K} - T_{\mathcal K_0}^*T_{\mathcal K_0}\|.
\end{align}

In a similar vein to $H_\infty$ control \cite{zames1981feedback}, we can introduce weighting functions to the regret-optimal control problem. Define three positive semidefinite operators: $\mathcal W_{\vs}\succeq 0$ is a weight on the state signal $\vs$, $\mathcal W_{\vu}\succeq 0$ is a weight for the control signal $\vu$, and $\mathcal W_{\vw}\succeq 0$ is a weight on the disturbance $\vw$. The generalized regret-optimal control problem with weights is defined through its sub-optimal problem
\begin{align}\label{eq:regret_weights}
&\vw^\ast \left(T_{\mathcal K}^*  \begin{pmatrix}
\mathcal W_{\vs} & 0 \\
0 & \mathcal W_{\vu}
\end{pmatrix} T_{\mathcal K} - T_{\mathcal K_0}^* \begin{pmatrix}
\mathcal W_{\vs} & 0 \\
0 & \mathcal W_{\vu}
\end{pmatrix}T_{\mathcal K_0}\right)\vw \nn\\
&\ \ \le \gamma^2 \vw^\ast \mathcal W_{\vw}\vw, \ \ \ \forall \vw.
\end{align}
The motivation for weight functions is from practical considerations, where one aims to attenuate the response of particular disturbances, control signals or state signals given particular system requirements. The corresponding operational problem minimizes the difference between weighted costs of the causal controller $\mathcal K$ and the non-causal controller $\mathcal K_0$ with respect to the weighted disturbance $\vw' = \mathcal W_{\vw}^{1/2}\vw $, where for an operator $\mathcal X \succeq 0$, we write $\mathcal X^{1/2}$ as the causal factor from the factorization $\mathcal X = \mathcal X^{1/2}\mathcal X^{\ast/2}$ whose inverse is bounded.

We proceed to show that \eqref{eq:regret_weights} can be formulated as a regret-optimal control problem. Define the cost operator
\begin{align}
    \overline{T}_{\mathcal K} &\triangleq \begin{pmatrix}
\mathcal W_{\vs}^{1/2} & 0 \\
0 & \mathcal W_{\vu}^{1/2}
\end{pmatrix}\begin{pmatrix}
\mathcal F\mathcal  K+\mathcal G \\
\mathcal K
\end{pmatrix} \mathcal W_{\vw}^{-/2} \nn\\
&= \begin{pmatrix}
\mathcal W_{\vs}^{1/2}\mathcal F\mathcal K\mathcal W_{\vw}^{-/2} + \mathcal W_{\vs}^{1/2} \mathcal G \mathcal W_{\vw}^{-/2} \\
\mathcal W_{\vu}^{1/2} \mathcal K \mathcal W_{\vw}^{-/2}
\end{pmatrix},
\end{align}
and note that with the change of variables $\overline{\mathcal  K} = \mathcal W_{\vu}^{1/2} \mathcal  K \mathcal  W_{\vw}^{-/2}$, $\overline{\mathcal  F} = \mathcal  W_{\vs}^{1/2}\mathcal F \mathcal W_{\vu}^{-/2}$, and $\overline{\mathcal G} = \mathcal W_{\vs}^{1/2} \mathcal G \mathcal W_{\vw}^{-/2},$
we obtain the regret-optimal control problem \looseness=-1
\begin{align}\label{eq:regret_modified}
    \inf_{\overline{\mathcal K}} \|\overline{T}_{\overline{\mathcal K}}^\ast\overline{T}_{\overline{\mathcal K}} - \overline{T}_{\mathcal K_0}^\ast\overline{T}_{\mathcal K_0}\|
\end{align}
with the modified system $(\overline{\mathcal F},\overline{\mathcal G})$. The advantage of this reduction is that the regret-optimal control problem can be solved optimally \cite{sabagFIACC,sabag2021regret}. By letting $\overline{\mathcal  K}$ to be the optimal solution to \eqref{eq:regret_modified}, the optimal controller in \eqref{eq:regret_weights} is computed as
\begin{align}\label{eq:weighted_controller}
\mathcal K&= \mathcal W_{\vu}^{-/2} \overline{\mathcal  K} \mathcal  W_{\vw}^{1/2}.
\end{align}


The relation between the competitive-ratio control and regret-optimal control problems becomes transparent. If we choose in the regret-optimal control problem the weights $\mathcal W_{\vs} = \mathcal W_{\vu}$ as the identity operator, and $\mathcal W_{\vw} = T_{\mathcal K_0}^\ast T_{\mathcal K_0}$, then it becomes the optimal competitive-ratio control problem. This demonstrates the flexibility of the generalized regret-optimal control framework presented here along with its advantageous explicit and optimal solution. On the other hand, note that one should compute explicitly the factorization $\mathcal W_{\vw}^{1/2}\mathcal W_{\vw}^{\ast/2}= \mathcal W_{\vw}$ such that $\mathcal W_{\vw}^{1/2}$ is causal and its inverse is bounded.

\begin{table*}[t]

  \caption{Performance of the controllers in different control systems. The highlighted values indicate the smallest value for each metric, i.e. column, excluding the non-causal controller. Each controller achieves the optimal performance, the lowest value, corresponding to their design metric \eqref{eq:ex_TkFrob}-\eqref{eq:ex_comp}. }
  \label{table_all}
\resizebox{\textwidth}{!}{{\begin{tabular}{ c|
c c c c | c c c c | c c c c |
}
  &
  \multicolumn{4}{|c|}{\textbf{HE1\cite{leibfritz2003description}}} & \multicolumn{4}{|c|}{\textbf{AC12\cite{leibfritz2003description}}} & \multicolumn{4}{|c|}{\textbf{HE4\cite{leibfritz2003description}}} \\ \hline
                  & $\mathbf{\|T_K\|_F^2}$ & $\mathbf{\|T_K\|^2}$ & \textbf{Regret} & \textbf{Comp-Ratio}& $\mathbf{\|T_K\|_F^2}$ & $\mathbf{\|T_K\|^2}$ & \textbf{Regret} & \textbf{Comp-Ratio} & $\mathbf{\|T_K\|_F^2}$ & $\mathbf{\|T_K\|^2}$ & \textbf{Regret} & \textbf{Comp-Ratio}\\
                  \hline
 Noncausal &
   $0.40\!\times\!10^0$ & $8.99\!\times\!10^1$ & 0 & 1 & $6.65\!\times\!10^2$ & $8.29\! \times\!10^4$ & 0& 1& $5.16\times10^3$ & $3.62\times10^5$ &0& 1\\

  $\mathcal H_2$ &
  $\mathbf{1.09\!\times\!10^0}$  &  $3.11\!\times\!10^2$  & $2.21\!\times\!10^2$ & $3.46\!\times\!10^0$ &  $\mathbf{2.76 \!\times\!10^3}$ & $1.92\!\times\!10^5$ & $1.27\!\times\!10^5$ & $7.95\!\times\!10^2$ & $\mathbf{3.24\!\times\!10^4}$ & $1.95\!\times\!10^6$ & $1.62\!\times\!10^6$ & $6.79\!\times\!10^2$ \\

 $\mathcal H_\infty$ &
 $1.31\!\times\!10^2$ & $\mathbf{1.31\!\times\!10^2}$  & $1.31\!\times\!10^2$ &  $1.19\!\times\!10^5$ & $8.29\!\times\!10^4$ & $\mathbf{8.29\!\times\!10^4}$ & $8.29\!\times\!10^4$ & $2.78\!\times\!10^6$ & $6.70\!\times\!10^5$ & $\mathbf{6.59\!\times\!10^5}$  & $6.59\!\times\!10^5$  & $2.67\!\times\!10^6$  \\

 Regret-optimal &
   $7.19\!\times\!10^1$ &  $1.61\!\times\!10^2$ &  $\mathbf{7.23\!\times\!10^1}$ & $6.40 \!\times\! 10^4$ & $3.99 \!\times\! 10^4$ & $1.21\!\times\!10^5$ & $\mathbf{3.97\!\times\!10^4}$ & $1.31\!\times\!10^6$ & $5.01\!\times10^5$ & $8.34\!\times10^5$ & $\mathbf{4.94\!\times10^5}$ &   $1.96\!\times10^6$
  \\
  $\text{CR-optimal}$ &
 $1.15\!\times\!10^0$ & $2.79\!\times\!10^2$  & $1.89\!\times\!10^2$  & $\mathbf{3.13\!\times\!10^0}$  & $4.50\!\times\!10^3$  & $5.10\!\times\!10^5$ & $4.37\!\times\!10^5$  & $\mathbf{7.74\!\times\!10^0}$ & $5.66\!\times\!10^4$ & $4.26\!\times\!10^6$ & $3.89\!\times\!10^6$ & $\mathbf{1.18\!\times\!10^1}$ \\

 \hline \hline &   \multicolumn{4}{|c|}{\textbf{REA1\cite{leibfritz2003description}} } &
  \multicolumn{4}{|c|}{\textbf{WEC1\cite{leibfritz2003description}}} &  \multicolumn{4}{|c|}{\textbf{AGS\cite{leibfritz2003description}}} \\ \hline
                  & $\mathbf{\|T_K\|_F^2}$ & $\mathbf{\|T_K\|^2}$ & \textbf{Regret} & \textbf{Comp-Ratio}& $\mathbf{\|T_K\|_F^2}$ & $\mathbf{\|T_K\|^2}$ & \textbf{Regret} & \textbf{Comp-Ratio} & $\mathbf{\|T_K\|_F^2}$ & $\mathbf{\|T_K\|^2}$ & \textbf{Regret} & \textbf{Comp-Ratio}\\
                  \hline
      Noncausal &
  $5.18\!\times\!10^1$ &  $2.05\!\times\!10^3$ & 0 & 1 & $4.23\!\times\!10^3$  & $4.01\! \times\!10^5$ & 0& 1& $4.31\times10^3$ & $6.68\times10^5$ &0& 1\\

  $\mathcal H_2$ &
  $\mathbf{2.62\!\times\!10^2}$   & $1.46\!\times\!10^4$ & $1.26\!\times\!10^4$  & $1.19\!\times\!10^1$ &  $\mathbf{1.04\!\times\!10^4}$ & $9.59\!\times\!10^5$ & $5.65\!\times\!10^5$ & $3.52\!\times\!10^2$ & $\mathbf{4.53\!\times\!10^3}$ & $6.69\!\times\!10^5$ & $7.79\!\times\!10^4$ & $3.09\!\times\!10^0$ \\

 $\mathcal H_\infty$ &
 $4.40\!\times\!10^3$  & $\mathbf{4.36\!\times\!10^3}$  & $4.36\!\times\!10^3$  & $1.78\!\times\!10^4$  & $4.20\!\times\!10^5$ & $\mathbf{4.19\!\times\!10^5}$ & $4.19\!\times\!10^5$ & $1.61\!\times\!10^6$ & $4.54\!\times\!10^3$ & $\mathbf{6.68\!\times\!10^5}$ & $5.17\!\times\!10^4$ & $3.16\!\times\!10^0$  \\

  Regret-optimal &
  $3.38\!\times\!10^3$ & $5.30\!\times\!10^3$ & $\mathbf{3.32\!\times\!10^3}$ & $1.35\!\times\!10^4$ &                $1.59 \!\times\! 10^5$  & $5.53\!\times\!10^5$ & $\mathbf{1.56\!\times\!10^5}$  & $5.70\!\times\!10^5$  & $2.49\!\times10^4$ & $6.69\!\times10^5$ & $\mathbf{2.06\!\times10^4}$  & $8.20\!\times\!10^4$
  \\

  $\text{CR-optimal}$ &
 $2.93\!\times\!10^2$ & $1.76\!\times\!10^4$ & $1.56\!\times\!10^4$ & $\mathbf{9.11\!\times\!10^0}$ & $2.93\!\times\!10^4$  & $2.51\!\times\!10^6$  & $2.11\!\times\!10^6$  & $\mathbf{3.82\!\times\!10^1}$  & $4.61\!\times\!10^3$ & $6.69\!\times\!10^5$ & $2.04\!\times\!10^5$ & $\mathbf{2.86\!\times\!10^0}$ \\ \hline
\end{tabular}}}
\end{table*}

\section{Numerical Examples} \label{sec:numerical}
In this section, we present the performance of the optimal competitive-ratio, regret,  $\mathcal H_2$, and $\mathcal H_\infty$ controllers for different systems in different disturbance regimes. First, we demonstrate the frequency domain evaluations to compare the performance of the controllers. We then study their time-domain behavior with different disturbance sequences.

\subsection{Frequency-domain}\label{subsec:ex_freq}

Recall the transfer (cost) operator $T_{\mathcal K}$ given in \eqref{transfer_operator}, which maps the disturbance sequence $\vw$ to the sequences $\vs$ and $\vu$. This operator governs the performance of any linear controller. In particular, the controllers discussed in this work aim to control different metrics of $T_{\mathcal K}$ across all range of disturbances. To this end, consider the transfer function representation of this operator in the $z$-domain:
\begin{align}\label{eq:ex_TKz}
T_{K}(z) &= \left[\begin{array}{c} F(z) K(z)+G(z) \\ K(z) \end{array} \right].
\end{align}

$\mathcal H_2$ controller minimizes the Frobenius norm of $T_{\mathcal K}(z)$, i.e.,
\begin{align}\label{eq:ex_TkFrob}
\|T_{\mathcal K}\|_F^2 &= \frac{1}{2\pi}\int_0^{2\pi}\mbox{trace}\left(T_{K}^*(e^{j\omega})T_{ K}(e^{j\omega})\right)d\omega,
\end{align}
and $\mathcal H_\infty$ controller minimizes the operator norm of $T_{\mathcal K}(z)$, i.e.,\looseness=-1
\begin{align}\label{eq:ex_Tknorm}
\|T_{\mathcal K}\|^2 &= \max_{0\leq\omega\leq 2\pi} \sigma_{\max}\left(T_{ K}^*(e^{j\omega})T_{K}(e^{j\omega})\right).
\end{align}
Similarly, the controllers that are constructed with the competitive design aim to minimize certain metrics of $T_{\mathcal K}(z)$ with respect to the transfer function of the non-causal controller $T_{\mathcal K_0}$. In particular, the regret-optimal controller minimizes
\begin{align}\label{eq:ex_regret}
&\left\|T_K^*T_K-T_{K_0}^*T_{K_0}\right\| \\
&= \max_{0\leq\omega\leq 2\pi}\sigma_{\max}\left(T_K^\ast (e^{j\omega})T_K(e^{j\omega})-T_{K_0}^\ast(e^{j\omega})T_{K_0}(e^{j\omega})\right), \nn
\end{align}
whereas the optimal competitive-ratio controller minimizes
\begin{equation}\label{eq:ex_comp}
    \max_{0\leq\omega\leq 2\pi}\! \!\sigma_{\max}\left(M^{-\ast}(e^{j\omega}) T_K^\ast (e^{j\omega})T_K(e^{j\omega}) M^{-1}(e^{j\omega})\right).
\end{equation}

To illustrate the performance of different controllers across the full range of input disturbances, we plot the metrics given in \eqref{eq:ex_TkFrob}-\eqref{eq:ex_comp} as a function of frequency for a randomly generated LTI system. For $n=4$ and $m=2$, we randomly generate all the system matrices, \textit{i.e.}, $A, B_u, B_w, Q, R$, such that $A$ is unstable but the pair $(A, B_u)$ is stabilizable. We construct
the optimal non-causal, $\mathcal H_2$, $\mathcal H_\infty$, regret, and competitive-ratio controllers and compute the transfer operators $T_K(e^{j\omega})$ for each of them. Figure \ref{fig:random} presents the performance of these controllers. \looseness=-1

\begin{figure}[t]
    \centering
    \includegraphics[scale=0.4]{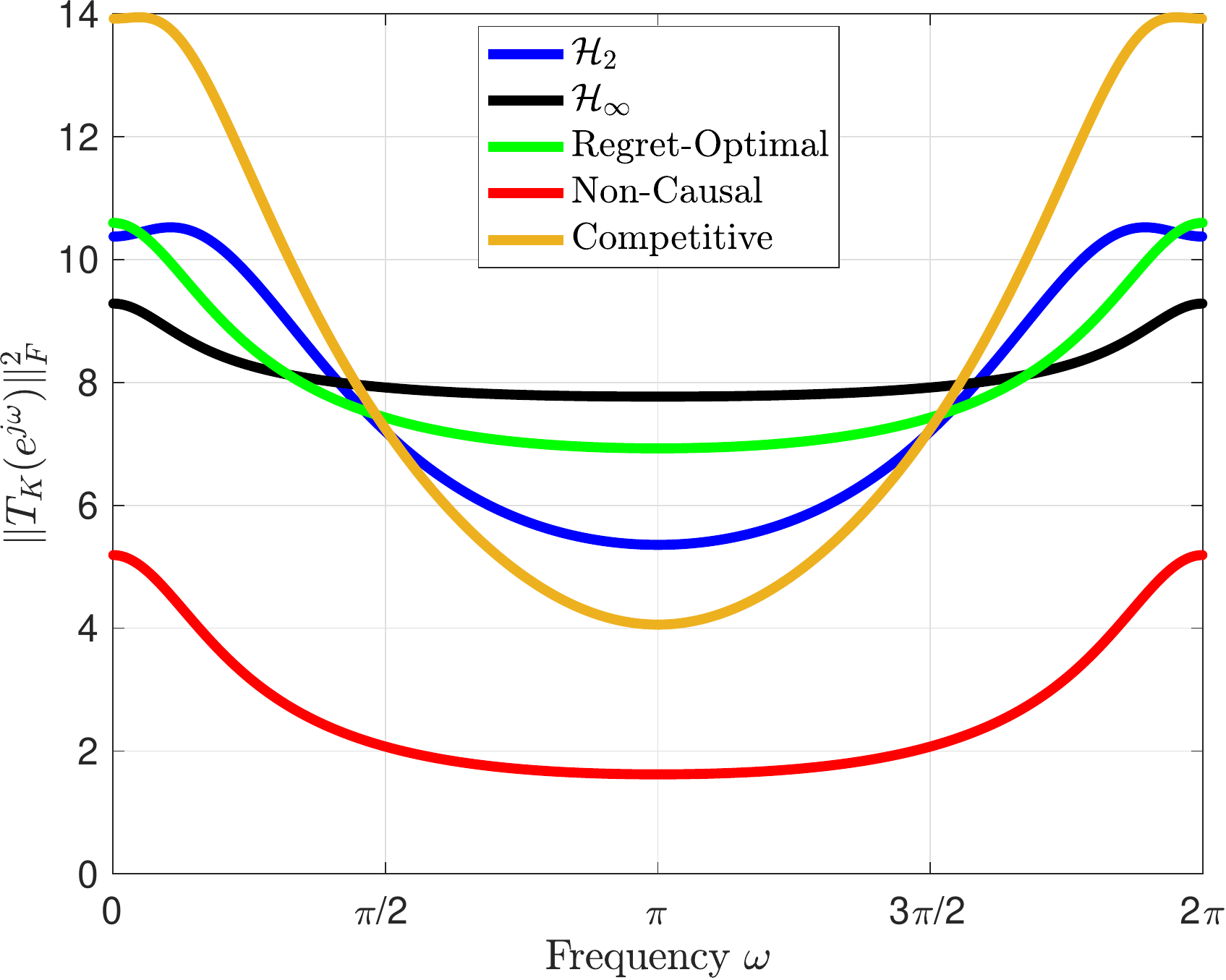}
    \includegraphics[scale=0.4]{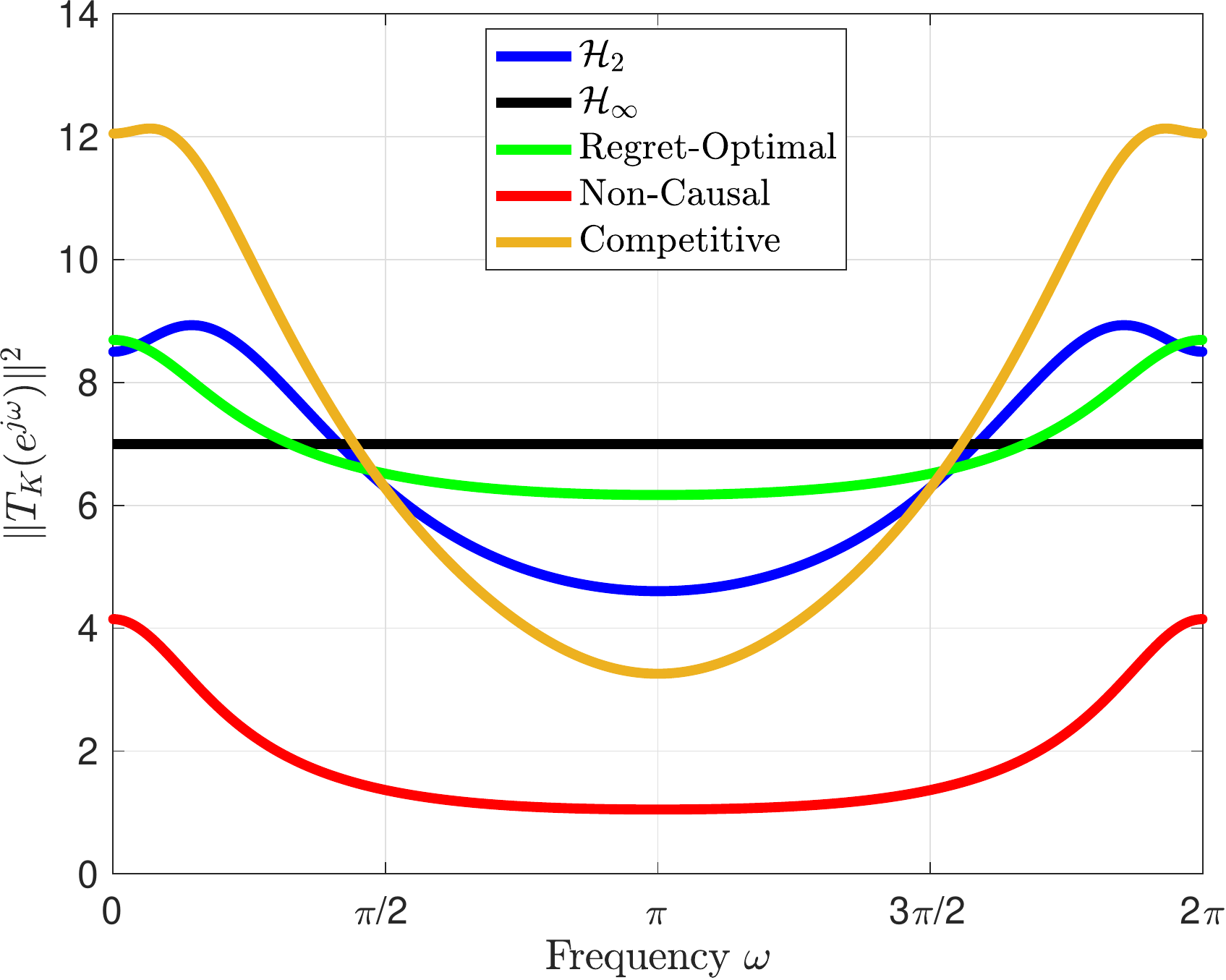}
    \includegraphics[scale=0.4]{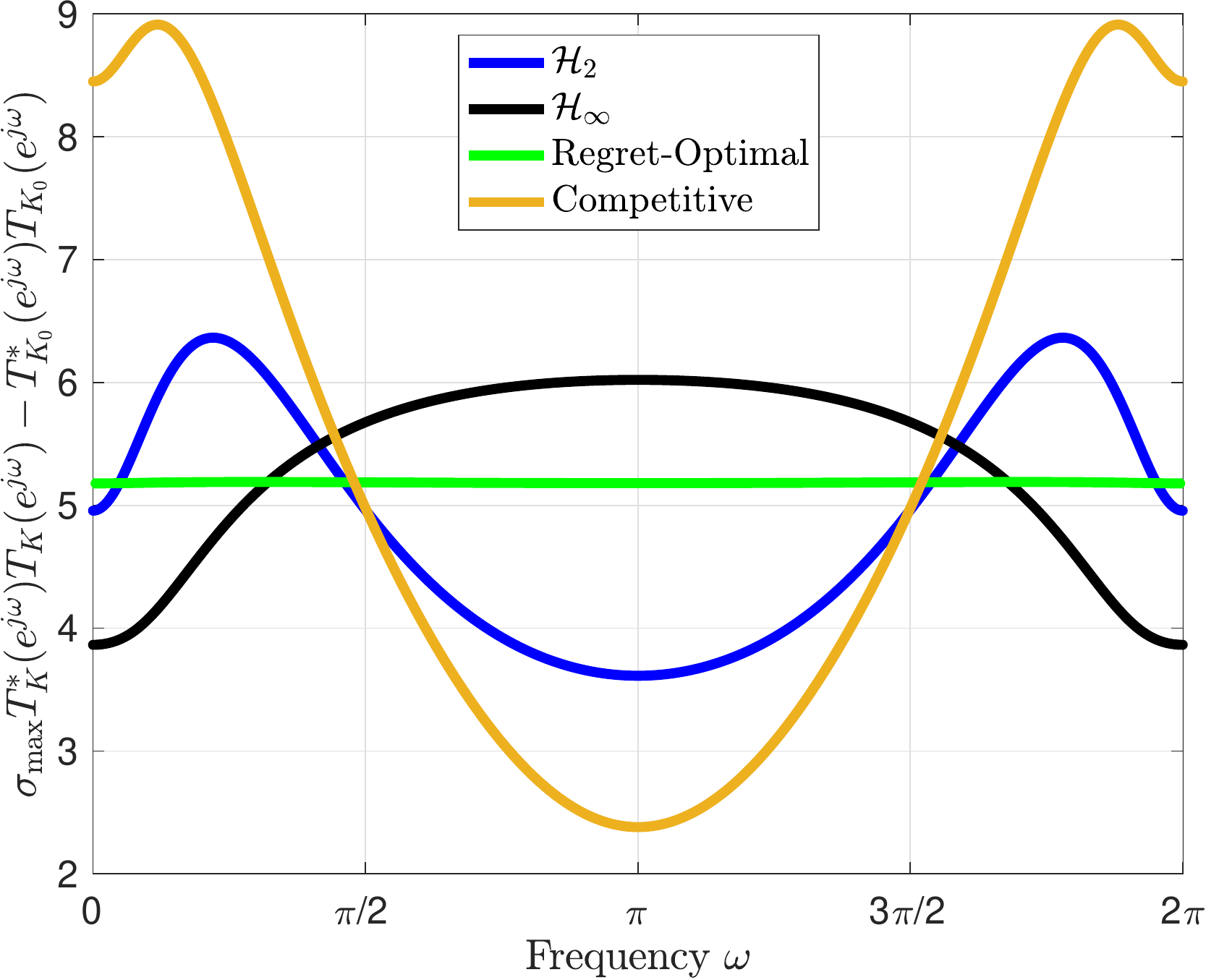}
    \includegraphics[scale=0.4]{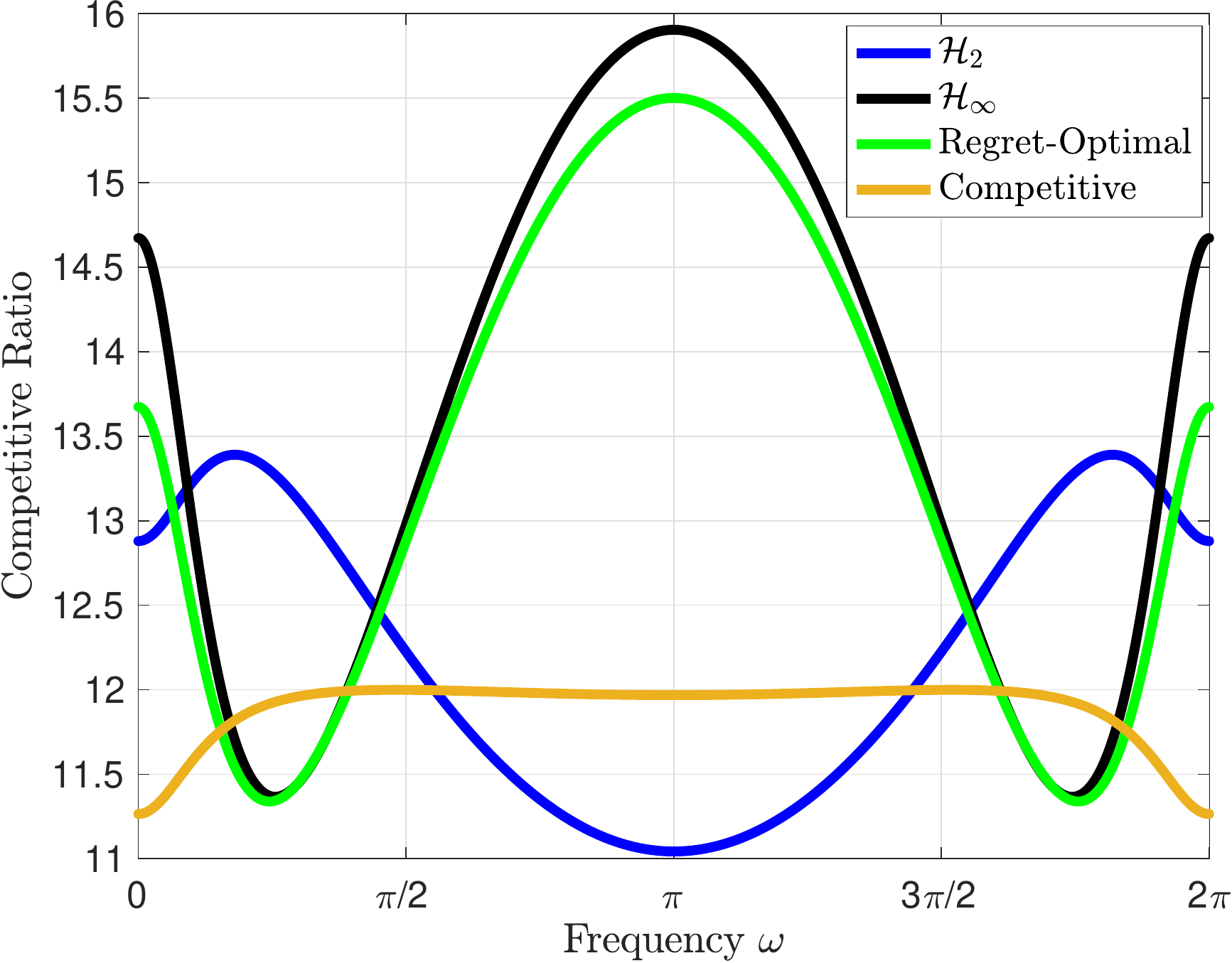}
    \caption{
    The different metrics in \eqref{eq:ex_TkFrob} - \eqref{eq:ex_comp} as a function of the frequency for the randomly generated linear dynamical system with six-dimensional state vector and two-dimensional control vector. The squared operator norm, squared Frobenius norm, regret, and competitive ratio for each controller are illustrated in the figures, respectively.} \
    \label{fig:random}
\end{figure}

As expected the non-causal controller outperforms
the four causal controllers in terms of Frobenius and operator norms across all frequencies as shown in the top two plots of Fig. \ref{fig:random}. The first figure demonstrates the per-frequency Frobenius norm of the cost operator for the different controllers. The $\mathcal H_2$ controller attains the minimal Frobenius norm, \textit{i.e.,} the area under the curve, as it is targeted to minimize \eqref{eq:ex_TkFrob}. The optimal competitive-ratio controller achieves a similar performance to the optimal $\mathcal H_2$ controller, and outperforms the $\mathcal H_\infty$ and the regret-optimal controllers. However, in doing so, it sacrifices the worst-case performance and so has a relatively large values for at the low frequencies as shown in the top right plot of Fig. \ref{fig:random}. As designed to minimize \eqref{eq:ex_Tknorm}, the $\mathcal H_\infty$ controller attains the smallest peak of the per-frequency operator norm. Notice that since the competitive-ratio controller is designed to minimize the ratio with respect to the non-causal controller, it follows the non-causal controller closely where the non-causal controller performs well, and sacrifices the performance where the non-causal controller does not achieve good performance. Conversely, the regret-optimal controller follows the performance of the non-causal controller uniformly (almost-constant distance from the non-causal controller) across all frequencies by minimizing its largest deviation from the latter. This demonstrates the different behavior of these controllers due to their design strategies.


The bottom two plots demonstrate the regret and competitive-ratio performances of the causal controllers per frequency. The peaks for each controller in these plots correspond to the regret and the competitive-ratio of each controller respectively, \textit{i.e.}, the values of \eqref{eq:ex_regret} and \eqref{eq:ex_comp} for the transfer functions of respective controllers. Consistent with our theoretical claims, the competitive controllers obtain the best performance in their corresponding metric. In particular, while the regret-optimal controller attains a uniform additive deviation from the non-causal controller, the optimal competitive-ratio controller attains a uniform multiplicative deviation from this universal benchmark. This shows that the competitive control design approaches produce distinctly different controller characteristics depending on the design metric.\looseness=-1

While the previous example was for illustration purposes, we proceed to show that the described behavior of the controllers is maintained in practical systems as well. In particular, we study six linear time-invariant models~\cite{leibfritz2003description}: two helicopter models (HE1 and HE4), an aircraft model (AC12), a chemical reactor model (REA1), a wind energy conversion system (WEC1), and an automobile gas turbine model (AGS). These models cover a wide range of applications and have various state and control input dimensions ranging from four-state and three-input dimensions of (AC12), up to eight-state and four-input dimensions of multi-purpose helicopter dynamics (HE4). For further details of these systems please refer to \cite{leibfritz2003description}. The norms evaluation for these systems is given in Table \ref{table_all}. It can be observed that the observations from Fig \ref{fig:random} hold for these practical control systems.

\subsection{Time-domain evaluation}\label{subsec:ex_time}

In this section, we examine the time-domain performance of these controllers across wide range of disturbances. To this end, we consider the control task of stabilization (hovering) of a twin-engined multi-purpose helicopter model (HE4) with linearized dynamics using state-feedback \cite{leibfritz2003description}. The task is noted as \textit{full-authority control}, where the controller has total control over the blade angles of the main and tail rotors, Section 12.2.2 of~\cite{skogestad2007multivariable}. The control system has eight dimensional state: Pitch attitude, roll attitude, roll rate, pitch rate, yaw rate, and three dimensional velocities; while the control has four dimensions which aim to affect the lift, longitudinal and lateral motion, as well as avoid spinning. For this dynamical system, we construct the optimal competitive-ratio, regret, $\mathcal H_2$, and $\mathcal H_\infty$ controllers. In all experiments, we run 30 independent trials and present the mean average control cost over time.

Before presenting time-domain results, we consider the logarithm of the operator norm of $T_\mathcal{K}$, \eqref{eq:ex_Tknorm} for all controllers in Figure~\ref{fig:HE4_freq}. This figure depicts the performance of each controller with respect to all possible frequency of sinusoidal disturbances. Notice that the sinusoidal disturbances around zero frequency result drastically varying behavior for each controller (the y-axis is in logarithmic scale) which we investigate further in our time-domain evaluations with these specific frequencies. Before discussing the effect of sinusoidal disturbances, we consider the cost attained by these controllers with standard Gaussian noise, $w_t \sim \mathcal{N}(0,I)$, in Figure \ref{fig:time_gauss}. As expected, the $\mathcal H_2$ controller outperforms the other causal controllers. The optimal competitive-ratio controller has a comparable performance and significantly outperforms other controllers. Notice that the average costs of each controller match the Frobenius norm of their respective transfer operators presented in Table \ref{table_all}, further verifying our theoretical claims.

\begin{figure}
     \centering
     \begin{subfigure}[b]{0.45\textwidth}
         \centering
         \includegraphics[width=\textwidth]{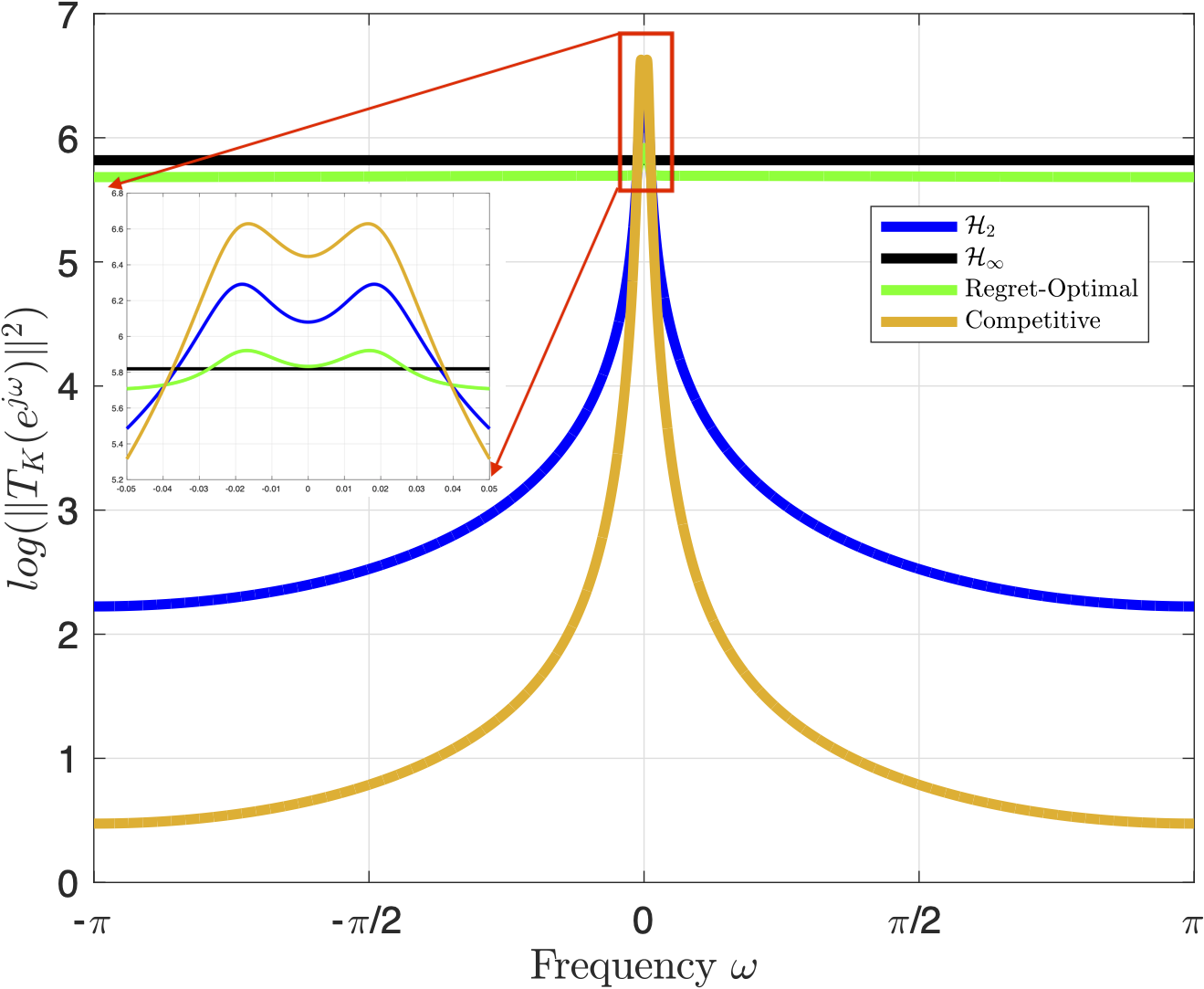}
         \caption{The operator norm for HE4}
         \label{fig:HE4_freq}
     \end{subfigure}
     \begin{subfigure}[b]{0.45\textwidth}
         \centering
         \includegraphics[width=\textwidth]{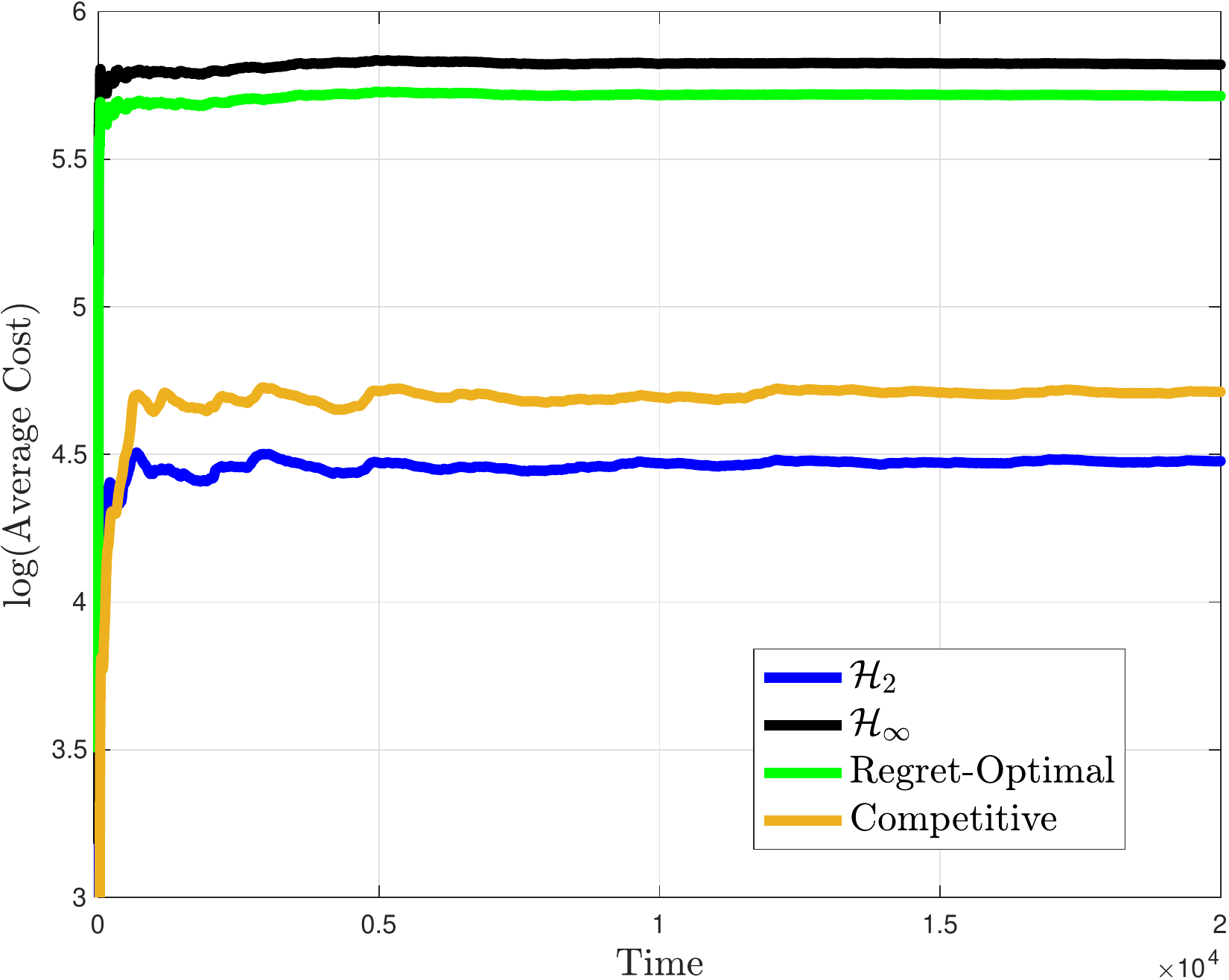}
         \caption{$w_t \sim \mathcal{N}(0,I)$}
         \label{fig:time_gauss}
     \end{subfigure}
     \begin{subfigure}[b]{0.45\textwidth}
         \centering
         \includegraphics[width=\textwidth]{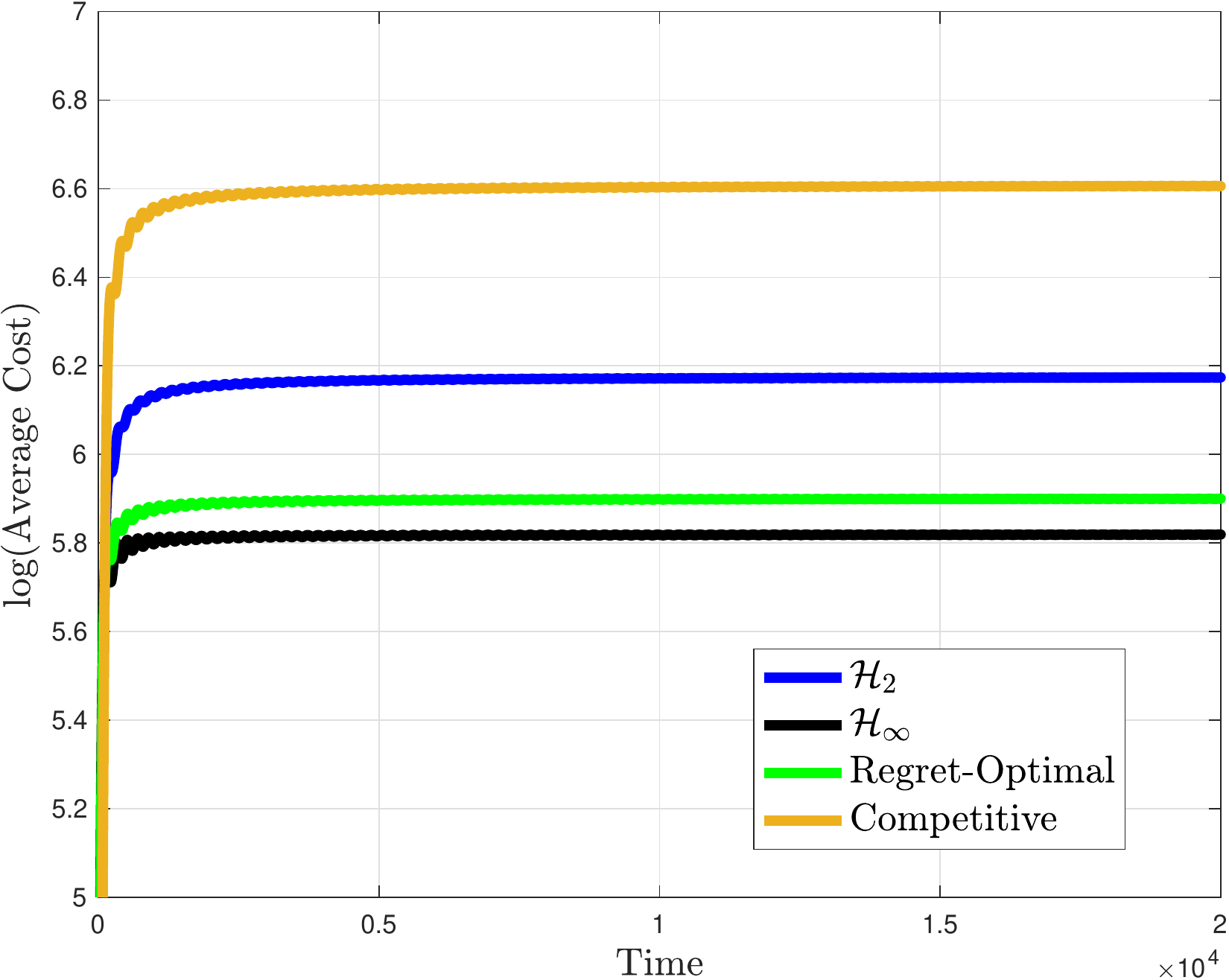}
         \caption{$w_t = \sin(0.016t)$}
         \label{fig:time_sine_low}
     \end{subfigure}
     \begin{subfigure}[b]{0.45\textwidth}
         \centering
         \includegraphics[width=\textwidth]{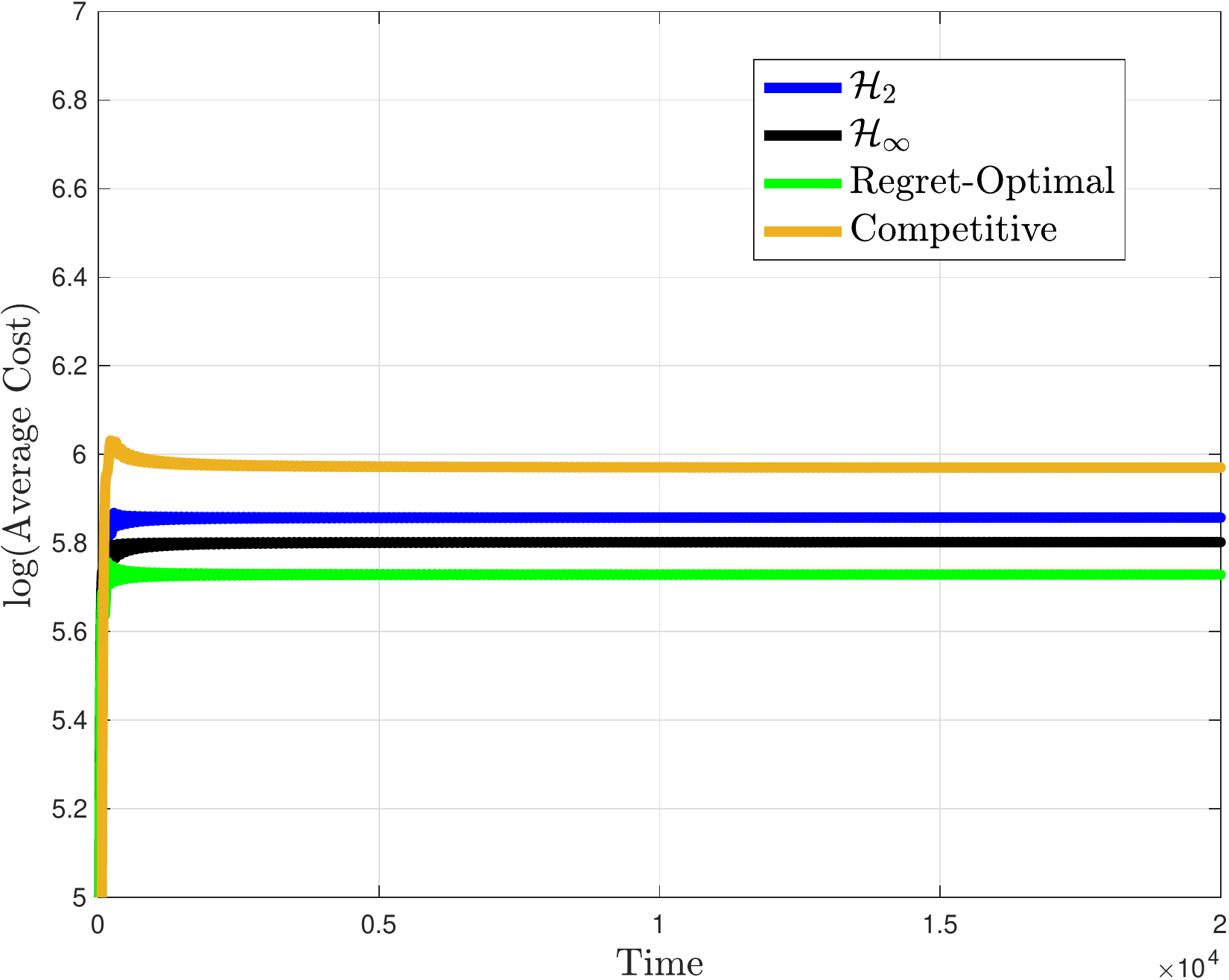}
         \caption{$w_t = \sin(0.034t)$}
         \label{fig:time_sine_med}
     \end{subfigure}
     \begin{subfigure}[b]{0.45\textwidth}
         \centering
         \includegraphics[width=\textwidth]{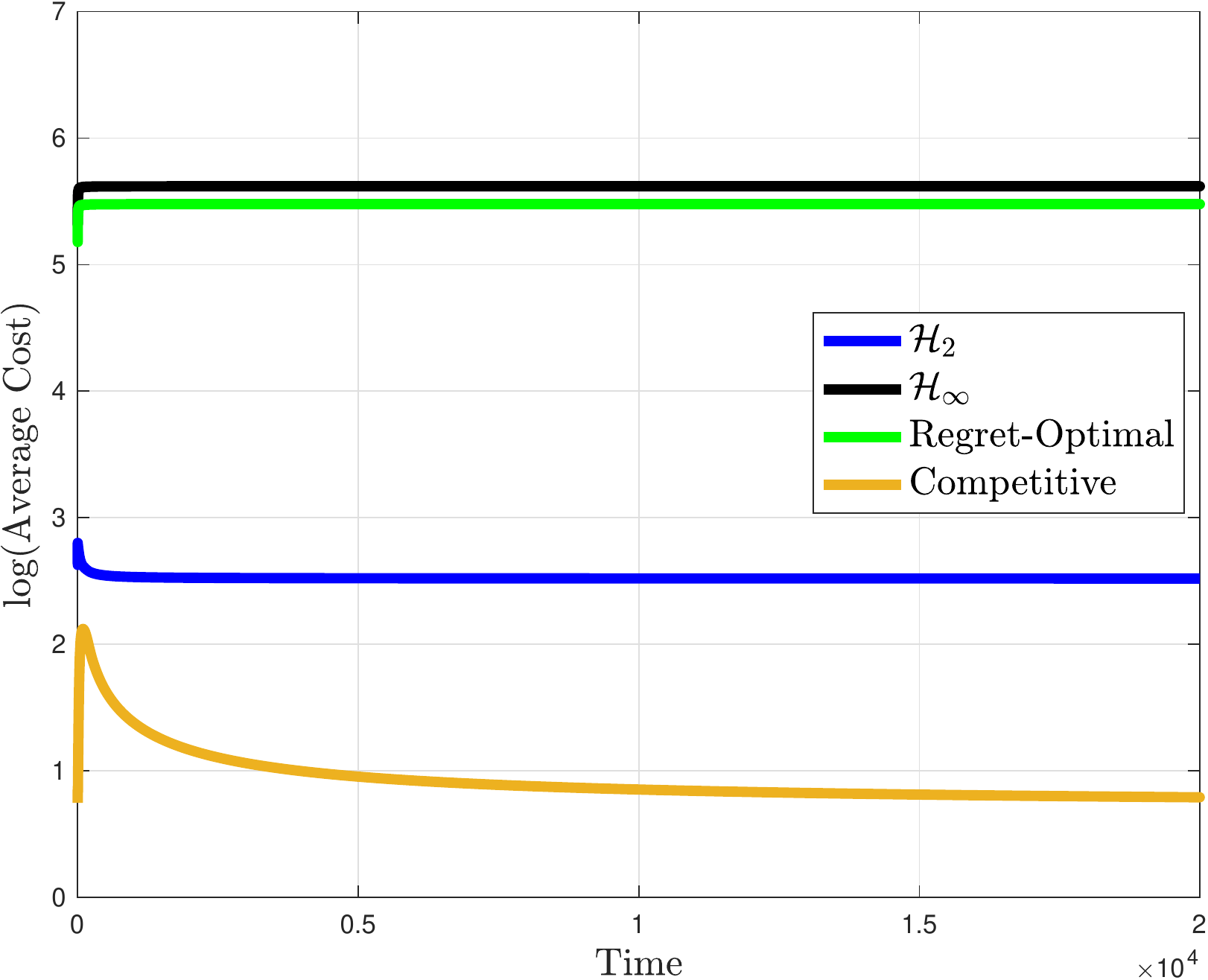}
         \caption{$w_t = \sin(\frac{\pi}{2}t)$}
         \label{fig:time_sine_high}
     \end{subfigure}
      \begin{subfigure}[b]{0.45\textwidth}
         \centering
         \includegraphics[width=\textwidth]{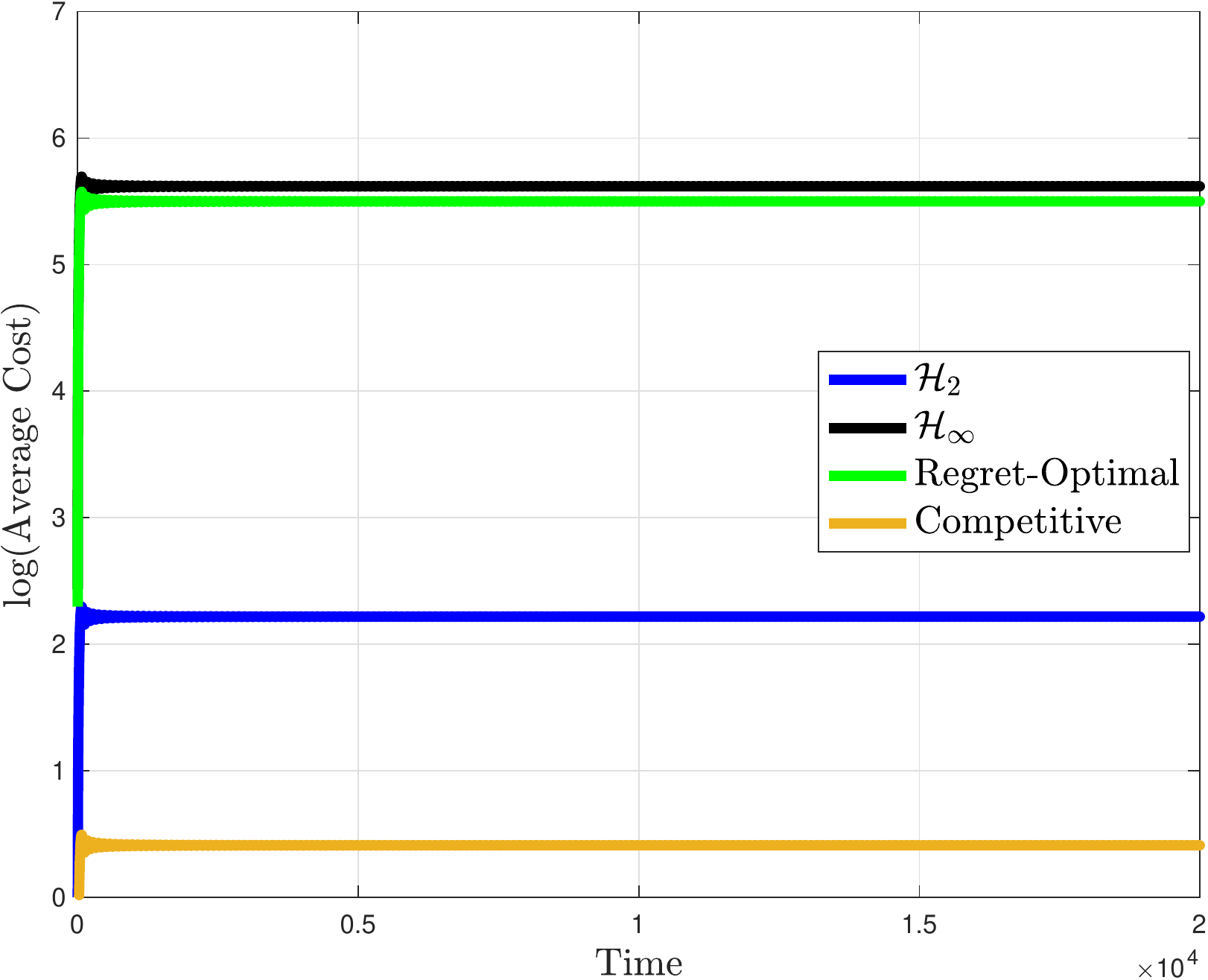}
         \caption{$w_t = \sin(0.99\pi t)$}
         \label{fig:time_sine_highest}
     \end{subfigure}
        \caption{Time Domain Performance of All Controllers Under Different Noise Disturbances}
        \label{fig:three graphs}
\end{figure}

Next, we evaluate the dynamical system under various sinusoidal disturbances. When we set $w_t = \sin(0.016t)$, \textit{i.e.}, the disturbance frequency that yields the worst performance for the optimal competitive-ratio controller (Figure \ref{fig:HE4_freq}), we observe that the $\mathcal{H}_\infty$ outperforms other controllers and the optimal competitive-ratio controller attains the largest average cost as expected in Figure \ref{fig:time_sine_low}. From Figure \ref{fig:HE4_freq}, we observe that for disturbances with frequency within $[0.03, 0.04]$, the regret-optimal controller achieves the lowest operator norm among all causal controllers, followed by the $\mathcal{H}_\infty$ controller. This phenomenon is depicted in the time-domain evaluation  of these controllers under the disturbance of $w_t = \sin(0.034t)$ in Figure \ref{fig:time_sine_med}. Furthermore, Figure \ref{fig:HE4_freq} shows that as the frequency of disturbances increases, the optimal competitive-ratio controller starts to outperform other controllers significantly. This behavior is again observed in time-domain evaluations with $w_t = \sin(\frac{\pi}{2}t)$ and $w_t = \sin = \sin(0.99\pi t)$ in Figures \ref{fig:time_sine_high} and \ref{fig:time_sine_highest} respectively. Finally, we would like to highlight that the mean average costs over time for all controllers exactly match the frequency domain evaluations for the operator norms of the corresponding transfer functions of the controllers. Overall, our time-domain evaluations show that the competitive design strategies attain significantly improved performances in certain disturbances on a large scale practical control system. Thus, we believe that these control design strategies are viable alternatives to the classical control design paradigms for improved performance. \looseness=-1

\section{Derivation of the state-space solutions}\label{sec:state-space}
This section provides the proof of our main results in Theorems \ref{th:SC} to \ref{th:scalar} for the state-space setting. We start by proving Theorem \ref{th:SC}, in which our main objective is to explicitly compute the optimal controller in \eqref{eq:op_optcon}
\begin{align}\label{eq:cont_freq}
    K(z)&= \Delta^{-1}(z)({K}'(z)+ {C}(z))M(z),
\end{align}
where all transfer functions follows from operational counterparts in \eqref{eq:op_optcon}, and $C(z)$ denotes the strictly causal part of $\Delta(z) K_0(z) M^{-1}(z)$. We present the lemmas that are required to compute each function in \eqref{eq:cont_freq} and then prove Theorem \ref{th:SC}. Proofs of the lemmas appear in the Appendix.

Recall that the transfer functions of the operators $\mathcal F$ and $\mathcal G$ in \eqref{eq:operator_sys} are given by
\begin{align}\label{eq:FG}
    F(z) &= Q^{1/2}\caus{A}B_u\nn\\
    G(z) &= Q^{1/2}\caus{A}B_w.
\end{align}
with $Q = Q^{1/2}Q^{1/2}$. That is, $F(z)$ and $G(z)$ map the control and the disturbance, respectively, to the state vector (weighted with $Q^{1/2}$).

The following lemma concerns with the canonical factorization that also appears in LQR ($\mathcal H_2$ control) and regret-optimal control. \looseness=-1
\begin{lemma}\label{lemma:Delta}
Assume that $(A,B_u)$ is stabilizable and $Q\succ0$. The transfer function $I+ F^\ast(z^{-\ast})F(z)$ can be factored as $\Delta^*(z^{-*})\Delta(z)$, where
\begin{align}\label{eq:delta_def}
  &\Delta(z) = (I + B_u^{*}PB_u)^{1/2} (I + K_{\emph{lqr}} \caus{A}B_u),
\end{align}
$P$ is the unique stabilizing solution to the Ricatti equation
\begin{align}
P&= A^\ast PA + Q - A^\ast PB_u(I + B_u^\ast P B_u)^{-1}B_u^\ast PA = 0,\nn
\end{align}
and $K_{\emph{lqr}} = (I + B_u^\ast P B_u)^{-1}B_u^\ast P A$. Furthermore, $\Delta^{-1}(z)$ is casual and bounded on the unit circle.
\end{lemma}
The proof of Lemma \ref{lemma:Delta} is standard, e.g., \cite[Lemma $2$]{sabag2021regret} and is thus omitted. The following lemma concerns with the dual canonical factorization in Lemma \ref{lemma:Delta}, and is required for the computation of $M(z)$ in Lemma \ref{lemma:M_general}.
\begin{lemma}\label{lemma:nabla}
Assume that $(A,Q)$ is detectable and $(A,B_u)$ is stabilizable. The transfer function $I+F(z)F^\ast(z^{-\ast})$ can be factored as $\nabla(z) \nabla^\ast(z^{-\ast}) $ with
\begin{align}
    \nabla(z)&= (Q^{1/2} \caus{A}K_T+I)R_T^{1/2},
\end{align}
where $T$ is the stabilizing solution of the Riccati equation
\begin{align}
    T &= ATA^\ast + B_u B_u^\ast - K_T (I+ Q^{1/2}TQ^{1/2})K_T^\ast,
\end{align}
$K_T = ATQ^{1/2}(I+ Q^{1/2}TQ^{1/2})^{-1}$, $R_T = R_T^{1/2}R_T^{1/2} = I+ Q^{1/2}TQ^{1/2}$, and $Q= Q^{1/2}Q^{1/2}$. Moreover, its inverse $\nabla(z)^{-1}$ is causal and bounded on the unit circle.
\end{lemma}
The following lemma is for the factorization of the clairvoyant cost operator.
\begin{lemma}\label{lemma:M_general}
Assume $B_w$ is a full-column rank. A causal transfer function $M(z)$ that satisfies the factorization
\begin{align}
  M^\ast(z^{-\ast})M(z) = T^\ast_{K_0}(z)T_{K_0}(z)
\end{align}
is given by
\begin{align}
    M(z)&= R_M^{1/2}(K_M \caus{A_T}B_w + I),
\end{align}
where $M$ is the stabilizing solution to the Riccati equation
\begin{align}
    M&= A_T^\ast M A_T + A_T^\ast Q^{1/2}R_T^{-1}Q^{1/2} A_T - K_M^\ast R_M^{-1}K_M\nn,
\end{align}
with $R_M = B_w^\ast Q^{1/2}R_T^{-1}Q^{1/2} B_w + B_w^\ast MB_w$ and $K_M = R_M^{-1}(B_w^\ast MA_T + B_w^\ast R_T^{-1} Q^{1/2}A_T)$. Moreover, its inverse
\begin{align}
    M^{-1}(z)&= (I - K_M \caus{A_M}B_w )R_M^{-/2}
\end{align}
exists and is bounded on the unit circle as $A_M = A_T - B_wK_M$ is a stable matrix.
\end{lemma}

The following lemma provides the decomposition of the transfer function $\Delta(z)K_0(z)M^{-1}(z)$ into its strictly-causal and anticausal counterparts.

\begin{lemma}\label{lemma:decomposition}
The product of the transfer functions $\Delta(z)K_0(z)M^{-1}(z) = -\Delta^{-\ast}(z^{-\ast}) F^\ast(z^{-\ast}) G(z) M^{-1}(z)$ can be written as the sum of
\begin{align}\label{eq:lemma_decomp_SC}
    A(z)&= - z^{-1} (I + B_u^\ast PB_u)^{-/2} B_u^\ast \anti{A_K^\ast}\nn\\
    &\ \ \ \cdot (P-A_K^\ast U)B_wR_M^{-/2}\nn\\
    C_1(z)&= - (I + B_u^\ast PB_u)^{-/2} B_u^\ast PA\caus{A} B_w M^{-1}(z)\nn\\
    C_2(z)&\mspace{-1mu}= \mspace{-1mu}(I \mspace{-4mu}+\mspace{-4mu} B_u^\ast PB_u)^{-/2} B_u^\ast U \caus{A_M} B_wR_M^{-/2},
\end{align}
where $U$ solves the Sylvester equation $U \!=\! A_K^\ast U A_M \!+\! P B_wK_M$. \looseness=-1
\end{lemma}
Note that $A(z)$ is a (stable) anti-causal transfer function while $C_1(z)$ and $C_2(z)$ are strictly causal transfer functions.

\begin{lemma}\label{lemma:our_nehari}
{The optimal solution} to the Nehari problem with the anticausal transfer function $z{A}(z)$ (given in \eqref{eq:lemma_decomp_SC}) is
\begin{align}\label{eq:sol_Nehari}
& {K}'(z) \\
&\ =
- (I + B_u^\ast PB_u)^{-/2}B_u^\ast \Pi \caus{F_\gamma} K_\gamma R_M^{-/2},\nn
\end{align}
where
\begin{align}\label{eq:th_Nehari_gain}
    K_\gamma &=  (I - A_K Z_\ast A_K^\ast \Pi )^{-1}A_K Z_\ast (P-A_K^\ast U)B_w\\
    F_\gamma &= A_K - K_\gamma R_M^{-1}B_w^\ast (P- U^\ast A_K),\nn
\end{align}
$\Pi$ and ${Z}_\ast$ are given in \eqref{eq:lyapunov}.
\end{lemma}

We proceed to prove Theorem \ref{th:SC}.
\begin{proof}[Proof of Theorem \ref{th:SC}]
{In order to} obtain the optimal competitive ratio, we apply Theorem \cite[Th. 9]{sabag2021regret} with the strictly anti-causal transfer function $zA(z)$ in Lemma \ref{lemma:decomposition}. This results in that the optimal competitive ratio in Theorem \ref{th:SC} is $1+\lambda_{\max}(Z_1\Pi)$. We proceed to derive the optimal competitive-ratio controller in \eqref{eq:op_optcon}
\begin{align}\label{eq:proof_th_cont}
    {K}(z)&= \Delta^{-1}(z)(K'(z)+ C_1(z) + C_2(z))M(z).
\end{align}
where $C_1(z), C_2(z)$ are given in Lemma \ref{lemma:decomposition}, $M(z)$ is given in Lemma \ref{lemma:M_general}, and Lemma \ref{lemma:Delta} gives
\begin{align}
    \Delta^{-1}(z)&= (I + K_{\text{lqr}} \caus{A}B_u)^{-1} (I + B_u^\ast P B_u)^{-/2}\nn\\
    &= (I - K_{\text{lqr}} \caus{A_K}B_u) (I + B_u^\ast P B_u)^{-/2}.\nn
    \end{align}

We compute each product in \eqref{eq:proof_th_cont} separately. The first product in \eqref{eq:proof_th_cont} is simply
\begin{align}\label{eq:dec_first}
    \Delta^{-1}(z)C_1(z)M(z) &= - K_{\text{lqr}} \caus{A_K} B_w.
\end{align}
The second product in \eqref{eq:proof_th_cont} is
\begin{align}\label{eq:dec_second}
    &\Delta^{-1}(z)C_2(z)M(z)\nn\\
    &= (I - K_{\text{lqr}} \caus{A_K}B_u) (I + B_u^\ast P B_u)^{-1}\nn\\
    &\cdot B_u^\ast U \caus{A_M} B_w (I + K_M \caus{A_T}B_w)\nn\\
    &= (I - K_{\text{lqr}} \caus{A_K}B_u) (I + B_u^\ast P B_u)^{-1} \nn\\
    &\cdot B_u^\ast U \caus{A_M} (I + B_wK_M \caus{A_T})B_w\nn\\
    &=(I - K_{\text{lqr}} \caus{A_K}B_u)  \nn\\
    &\ \ \cdot (I + B_u^\ast P B_u)^{-1} B_u^\ast U \caus{A_T} B_w,
\end{align}
where the last equality follows from $\caus{A_M} = \caus{A_T} (I + B_wK_M\caus{A_T})^{-1}$. The last product in \eqref{eq:proof_th_cont} is \looseness=-1
\begin{align}\label{eq:dec_third}
    &\Delta^{-1}(z){K}'(z)M(z)\nn\\
    &= - (I - K_{\text{lqr}} \caus{A_K}B_u) (I + B_u^\ast PB_u)^{-1}\nn\\
    &\cdot B_u^\ast \Pi \caus{F_\gamma} K_\gamma (K_M \caus{A_T}B_w + I).
\end{align}
Combining \eqref{eq:dec_first}-\eqref{eq:dec_third}, have
\begin{align}
    K(z)&=
    \begin{pmatrix}
    \Lambda U & - \Lambda \Pi & - K_{\text{lqr}}
    \end{pmatrix}\nn\\
    &\cdot \caus{\begin{pmatrix}
    A_T&0&0\\
     K_\gamma K_M &F_\gamma&0\\
    B_u\Lambda U & - B_u \Lambda \Pi&A_K
    \end{pmatrix}}
    \begin{pmatrix}
    B_w\\
    K_\gamma\\
    B_w
    \end{pmatrix}
\end{align}
with $\Lambda = (I + B_u^\ast P B_u)^{-1} B_u^\ast$. The equivalent time-domain state-space is
\begin{align}
    \begin{pmatrix}
    \xi^1_{t+1}\\
    \xi^2_{t+1}\\
    \xi^3_{t+1}
    \end{pmatrix} &= \begin{pmatrix}
    A_T&0&0\\
     K_\gamma K_M &F_\gamma&0\\
    B_u\Lambda U & - B_u \Lambda \Pi&A_K
    \end{pmatrix}
    \begin{pmatrix}
    \xi^1_{t}\\
    \xi^2_{t}\\
    \xi^3_{t}
    \end{pmatrix}  +
        \begin{pmatrix}
    B_w\\
    K_\gamma\\
    B_w
    \end{pmatrix}w_t \nn\\
    u_t&= \begin{pmatrix}
    \Lambda U & - \Lambda \Pi & - K_{\text{lqr}}
    \end{pmatrix}
    \begin{pmatrix}
    \xi^1_{t}\\
    \xi^2_{t}\\
    \xi^3_{t}
    \end{pmatrix}\nn\\
    &\triangleq \bar{H} \bar{\xi}_t - K_{\text{lqr}}\xi^3_{t}.
\end{align}

We proceed to show that $\xi_t^3$ is in fact the system state $x_t$ which leads to a significant simplification in the implementation of the controller. The equation of the last state is:
\begin{align}
    \xi^3_{t+1}&= B_u \bar{H} \bar{\xi}_t + A_K\xi^3_{t} + B_w w_t\nn\\
    &= B_u \bar{H} \bar{\xi}_t + (A-B_uK_{\text{lqr}})\xi^3_{t} + B_w w_t\nn\\
    &= A\xi^3_{t}  + B_u (\bar{H} \bar{\xi}_t -K_{\text{lqr}}\xi^3_{t}) + B_w w_t\nn\\
    &= A\xi^3_{t}  + B_u u_t + B_w w_t,
\end{align}
which is precisely the evolution of the system state. This completes the derivation of the controller in Theorem \ref{th:SC}.
\end{proof}
\subsection{Proof of Theorem \ref{th:square_SC}}
We start by presenting two technical lemmas. The first is for the simplified factorization of $M(z)$.
\begin{lemma}\label{lemma:M}
If $B_w$ is a square matrix, the causal transfer function
\begin{align}\label{eq:lemma_Msq}
    M(z)&= z\nabla^{-1}(z)G(z) \nn\\
    &= R_T^{-/2} Q^{1/2}(I + A_T\caus{A_T})B_w
\end{align}
satisfies $M^\ast(z)M(z) = T_{K_0}^\ast(z)T_{K_0}(z)$ and
its causal inverse
\begin{align}
    M^{-1}(z)&=  B_w^{-1}(I-z^{-1}A_T)Q^{-/2}R_T^{1/2}.
\end{align}
is bounded on the unit circle.
\end{lemma}
The next lemmas are for the simplified decomposition and the solution to the Nehari problem in the case of a square $B_w$.
\begin{lemma}\label{lemma:decomposition_sq}
If $B_w$ is a square matrix, the decomposition of the transfer function $-\Delta^{-\ast}(z^{-\ast}) F^\ast(z^{-\ast}) G(z) M^{-1}(z)$ for the strictly causal scenario is
\begin{align}\label{eq:lemma_decomposition_sq}
    \overline{A}(z)&= - z^{-1} (I + B_u^\ast PB_u)^{-/2} B_u^\ast  \anti{A_K^\ast}\nn\\
    &\cdot (P - A_K^\ast PA_T)Q^{-/2}R_T^{1/2} \nn\\
    \overline{C}(z)&= - (I + B_u^\ast PB_u)^{-/2} B_u^\ast
    P\caus{A}\nn\\
    &\cdot (A-A_T)Q^{-/2}R_T^{1/2}.
\end{align}
\end{lemma}

\begin{lemma}\label{lemma:our_nehari_sq}
If $B_w$ is a square matrix, the optimal solution to the Nehari problem with the anticausal transfer function $z\overline{A}(z)$ (given in \eqref{eq:lemma_decomposition_sq}) is
\begin{align}
    \lambda_{\max}(Z_1\overline{\Pi})
\end{align}
with $Z_1$ given in \eqref{eq:lyapunov} and $\overline{\Pi}$ solves
\begin{align}
    \overline{\Pi}&= A_k^\ast \overline{\Pi} A_k + (P - A_K^\ast PA_T)Q^{-/2}R_TQ^{-/2} (P - A_T^\ast P A_K),\nn
\end{align}
and can be achieved with the causal transfer function
\begin{align}\label{eq:KN_sq}
& \overline{K}'(z) \\
&\ =
- (I + B_u^\ast PB_u)^{-/2}B_u^\ast \overline{\Pi} \caus{\overline{F}_\gamma} \overline{K}_\gamma Q^{-/2}R_T^{1/2},\nn
\end{align}
where
\begin{align}\label{eq:th_Nehari_gain_sq}
    \overline{K}_\gamma &=  (I - A_K \overline{Z}_\ast A_K^\ast \overline{\Pi} )^{-1}A_K \overline{Z}_\ast (P - A_K^\ast PA_T)\\
    \overline{F}_\gamma &= A_K - \overline{K}_\gamma Q^{-/2}R_T^{1/2} ((P - A_K^\ast PA_T)Q^{-/2}R_T^{1/2})^\ast,\nn
\end{align}
and $\overline{Z}_\ast$ is obtained from \eqref{eq:lyapunov} with~$\gamma^2 = \lambda_{\max}(Z_1\overline{\Pi})$.
\end{lemma}
The proof of Lemma \ref{lemma:M} is proved in the appendix, while Lemmas \ref{lemma:decomposition_sq} and \ref{lemma:our_nehari_sq} follow from steps that are similar to Lemma \ref{lemma:decomposition} and Lemma \ref{lemma:our_nehari}, respectively, and are thus omitted for brevity.

\begin{proof}
We proceed to prove Theorem \ref{th:square_SC} by simplifying the controller
\begin{align}
    \overline{K}(z)&= \Delta^{-1}(z)( \overline{K}'(z)+ \overline{C}(z))M(z),
\end{align}
where $\overline{K}'(z)$, $\overline{C}(z)$ and $M(z)$ are given in \eqref{eq:KN_sq}, \eqref{eq:lemma_decomposition_sq} and \eqref{eq:lemma_Msq}, respectively.
For simplicity, we denote $\Lambda = (I + B_u^\ast P B_u)^{-1}  B_u^\ast$. The first product is:
\begin{align}
    &\Delta^{-1}(z)\overline{C}(z)M(z)\nn\\
    &= -  (I + K_{\text{lqr}} \caus{A}B_u)^{-1} \Lambda
    P \caus{A}A \nn\\
    &\cdot(I + A_T\caus{A_T})B_w\nn\\
    & +  (I + K_{\text{lqr}} \caus{A}B_u)^{-1} \Lambda
    P \caus{A}A_T \nn\\
    &\cdot z \caus{A_T}B_w\nn\\
    &= - (I + K_{\text{lqr}} \caus{A}B_u)^{-1} \Lambda PA \caus{A} \nn\\
    &\cdot(I + A_T\caus{A_T})B_w\nn\\
    & +  (I + K_{\text{lqr}} \caus{A}B_u)^{-1} \Lambda
    P(I+A\caus{A})\nn\\
    &\cdot A_T \caus{A_T}B_w\nn\\
    &= - (I + K_{\text{lqr}} \caus{A}B_u)^{-1} \Lambda PA\caus{A} B_w\nn\\
    & +  (I + K_{\text{lqr}} \caus{A}B_u)^{-1} \Lambda
    P A_T \caus{A_T}B_w\nn\\
    &= - K_{\text{lqr}} \caus{A_K} B_w\\
    & +  (I - K_{\text{lqr}} \caus{A_K}B_u) \Lambda
    P A_T \caus{A_T}B_w\nn
\end{align}
The second product is:
\begin{align}
    &\Delta^{-1}(z)\overline{K}'(z)M(z)\nn\\
    &= - (I - K_{\text{lqr}} \caus{A_K}B_u) \Lambda \overline{\Pi} \caus{\overline{F}_\gamma} \nn\\
    &\cdot\overline{K}_\gamma (I + A_T\caus{A_T})B_w
\end{align}
Putting it all together gives:
\begin{align}
    \overline{K}(z)&=     \begin{pmatrix}
    (R + B_u^\ast P B_u)^{-1}
    B_u^\ast P A_T & - \Lambda \overline{\Pi} & - K_{\text{lqr}}
    \end{pmatrix}\\
    &\cdot \caus{\begin{pmatrix}
    A_T&0&0\\
     \overline{K}_\gamma A_T &F_\gamma&0\\
    B_u \Lambda
    P A_T & - B_u \Lambda \overline{\Pi}&A_K
    \end{pmatrix}}
    \begin{pmatrix}
    I\\
    \overline{K}_\gamma\\
    I
    \end{pmatrix}B_w\nn
\end{align}
Now, let
\begin{align}
    \begin{pmatrix}
    \xi^1_{t+1}\\
    \xi^2_{t+1}\\
    \xi^3_{t+1}
    \end{pmatrix} &= \begin{pmatrix}
    A_T&0&0\\
     \overline{K}_\gamma A_T &F_\gamma&0\\
    B_u \Lambda
    P A_T & - B_u \Lambda \overline{\Pi}&A_K
    \end{pmatrix}
    \begin{pmatrix}
    \xi^1_{t}\\
    \xi^2_{t}\\
    \xi^3_{t}
    \end{pmatrix} \nn\\
    &\ +\begin{pmatrix}
    I\\
    \overline{K}_\gamma\\
    I
    \end{pmatrix}B_w w_t \\
    u_t&= \begin{pmatrix}
    \Lambda P A_T & - \Lambda \overline{\Pi} & - K_{\text{lqr}}
    \end{pmatrix}
    \begin{pmatrix}
    \xi^1_{t}\\
    \xi^2_{t}\\
    \xi^3_{t}
    \end{pmatrix}\\
    &\triangleq \bar{H} \bar{\xi}_t - K_{\text{lqr}}\xi^3_{t}
\end{align}
The equation of the last state is:
\begin{align}
    \xi^3_{t+1}&= B_u \bar{H} \bar{\xi}_t + A_K\xi^3_{t} + B_w w_t\nn\\
    &= B_u \bar{H} \bar{\xi}_t + (A-B_uK_{\text{lqr}})\xi^3_{t} + B_w w_t\nn\\
    &= A\xi^3_{t}  + B_u (\bar{H} \bar{\xi}_t -K_{\text{lqr}}\xi^3_{t}) + B_w w_t\nn\\
    &= A\xi^3_{t}  + B_u u_t + B_w w_t.
\end{align}
To summarize, the controller is:
\begin{align}
    u_t&= (R + B_u^\ast P B_u)^{-1}
    B_u^\ast \begin{pmatrix}
     P A_T & - \overline{\Pi}
    \end{pmatrix}
    \begin{pmatrix}
    \xi^1_{t}\\
    \xi^2_{t}
    \end{pmatrix} - K_{\text{lqr}}x_t
\end{align}
with
\begin{align}
    \begin{pmatrix}
    \xi^1_{t+1}\\
    \xi^2_{t+1}
    \end{pmatrix} &= \begin{pmatrix}
    A_T&0\\
     \overline{K}_\gamma A_T &F_\gamma
    \end{pmatrix}
    \begin{pmatrix}
    \xi^1_{t}\\
    \xi^2_{t}
    \end{pmatrix}  +\begin{pmatrix}
    I\\
    \overline{K}_\gamma
    \end{pmatrix}B_w w_t.
\end{align}
\end{proof}
\subsection{Proof of Theorem \ref{th:scalar}}
We start by computing the solutions to the Lyapunov equations in \eqref{eq:lyapunov} $Z_1= \frac{B_u^2}{(1+B_u^2 P)(1-A_K^2)}, \overline{\Pi}= P^2(Q^{-1}+P)(1-A_K^2)$, which result that the competitive ratio is
\begin{align}
    \text{Comp-Ratio}&= 1 + \frac{B_u^2P^2(Q^{-1}+P)}{1+B_u^2 P}.
\end{align}
We can then compute $\overline{Z}_\ast= \frac1{P^2(1-A_K^2)(Q^{-1}+P)}$, $\overline{K}_\gamma= \frac{A_K}{P(Q^{-1}+P)(1-A_K^2)}$, and $\overline{F}_\gamma= 0$.

Since $\overline{F}_\gamma= 0$, the controller can be written as
\begin{align}
    u_t = - K_{\text{lqr}} x_t + (R + B_u^\ast P B_u)^{-1}
    B_u^\ast s_t
\end{align}
with $s_{t+1}\triangleq PA_K\xi^1_{t+1} - \overline{\Pi}\xi^2_{t+1}$. We can now show that
    \begin{align}
    s_{t+1}&= (PA_K^2 - \overline{\Pi} \overline{K}_\gamma A_K)\xi^1_t + (PA_K - \overline{\Pi} \overline{K}_\gamma ) B_ww_t\nn\\
    &= 0,
\end{align}
where in the first equality we used
\begin{align}
    \xi^1_{t+1}&= A_K\xi^1_t + B_w w_t, \ \ \xi^2_{t+1}= \overline{K}_\gamma A_K\xi^1_t + \overline{K}_\gamma B_w w_t,\nn
\end{align}
and the second equality follows from substituting the explicit constants above.

\bibliography{ref}

\appendix
\section{Proofs of Lemmas $2-5$}

\begin{proof}[\textbf{Proof of Lemma \ref{lemma:nabla}}]
Recall that
\begin{align}
    &I+F(z)F^\ast(z^{-\ast})\nn\\
    &\ \ = I + Q^{1/2}\caus{A}B_uB_u^\ast \anti{A^\ast}Q^{1/2}.
    \end{align}
Writing it in a matrix form
\begin{align}
&\begin{pmatrix}
     Q^{1/2} \caus{A} & I
    \end{pmatrix}
    \begin{pmatrix}
    B_uB_u^\ast & 0\\
    0 & I
    \end{pmatrix}
    \begin{pmatrix}
    \anti{A^\ast}Q^{1/2} \\ I
    \end{pmatrix}\nn\\
&=\begin{pmatrix}
     Q^{1/2} \caus{A} & I
    \end{pmatrix}\nn\\
    &\ \ \cdot\begin{pmatrix}
    B_uB_u^\ast - T + ATA^\ast& ATQ^{1/2}\\
    Q^{1/2}TA^\ast & I + Q^{1/2}TQ^{1/2}
    \end{pmatrix}\nn\\
    &\ \ \cdot \begin{pmatrix}
    \anti{A^\ast}Q^{1/2} \\ I
    \end{pmatrix}.
\end{align}
Note that $R_T\triangleq I + Q^{1/2}TQ^{1/2}\succ0$ and, therefore, the middle matrix can be written as
\begin{align}
    &\begin{pmatrix}
    B_uB_u^\ast - T + ATA^\ast& ATQ^{1/2}\\
    Q^{1/2}TA^\ast & I + Q^{1/2}TQ^{1/2}
    \end{pmatrix}\nn\\
    &=\begin{pmatrix}
  ATQ^{1/2} R_T^{-1}\\
  I
    \end{pmatrix}
    R_T
     \begin{pmatrix}
    R_T^{-1}Q^{1/2}TA^\ast& I
    \end{pmatrix} \\
    &+
    \begin{pmatrix}
     B_uB_u^\ast - T + ATA^\ast -  ATQ^{1/2} R_T^{-1}Q^{1/2}TA^\ast&0\\
    0&0
    \end{pmatrix}\nn.
\end{align}
We now choose $T$ as a solution to the Riccati equation and obtain to obtain
\begin{align}
    \nabla(z)&= \begin{pmatrix}
     Q^{1/2} \caus{A} & I
    \end{pmatrix}
    \begin{pmatrix}
    K_T\\
    I
    \end{pmatrix}
    R_T^{1/2}\\
    &= (Q^{1/2} \caus{A}K_T+I)R_T^{1/2}
\end{align}
where $K_T \triangleq ATQ^{1/2} R_T^{-1}$. Furthermore, $T$ is chosen as the stabilizing solution to the Riccati equation which exists since $(A,Q^{1/2})$ is detectable and $(A,B_u)$ is stabilizable.
\end{proof}

\begin{proof}[\textbf{Proof of Lemma \ref{lemma:M_general}}]
We use Lemma \ref{lemma:nabla} for the factorization $\nabla(z)\nabla^\ast(z^{-\ast}) = I + F(z)F^\ast(z^{-\ast})$. First, we note that case where $B_w$ is squared, we can simply choose $M(z) = z\nabla^{-1}(z)G(z)$ which is causal and its inverse

For the case where $B_w$ is non-square, we need an additional factorization. Recall that by Lemma \ref{lemma:nabla}
\begin{align}\label{eq:zNablaG}
    z\nabla^{-1}(z)G(z) &= R_T^{-/2} Q^{1/2}(I + A_T\caus{A_T})B_w.
\end{align}
We start by writing $M^\ast(z^{-\ast}) M(z)$ it in a matrix form
\begin{align}\label{eq:popov_M}
    \begin{pmatrix}
     B_w^\ast \anti{A_T^\ast} & I
    \end{pmatrix} \Lambda \begin{pmatrix}
    \caus{A_T}B_w \\ I
    \end{pmatrix}
\end{align}
with
\begin{align}
    \Lambda&\triangleq
    \begin{pmatrix}
    A_T^\ast Q^{1/2}R_T^{-1}Q^{1/2} A_T & A_T^\ast Q^{1/2}R_T^{-1} Q^{1/2}B_w\\
    (\cdot)^\ast & B_w^\ast Q^{1/2}R_T^{-1}Q^{1/2} B_w
    \end{pmatrix}.
\end{align}
For any Hermitian $M$, let
\begin{align}
    \Lambda(M) &= \begin{pmatrix}
      A_T^\ast M A_T - M + A_T^\ast Q^{1/2}R_T^{-1}Q^{1/2} A_T & K_M^\ast R_M \\
    R_MK_M & R_M
    \end{pmatrix}\nn\\
    K_M &= R_M^{-1}(B_w^\ast MA_T + B_w^\ast Q^{1/2} R_T^{-1} Q^{1/2}A_T)\nn\\
    R_M &= B_w^\ast Q^{1/2}R_T^{-1}Q^{1/2} B_w + B_w^\ast MB_w.
\end{align}
It can be directly verified that replacing $\Lambda$ with $\Lambda(M)$ in the Popov function \eqref{eq:popov_M} does not change the product.

We now choose $M$ such that the Schur complement of $\Lambda(M)$ is equal to zero:
\begin{align}\label{eq:RiccM}
    M&= A_T^\ast M A_T + A_T^\ast Q^{1/2}R_T^{-1}Q^{1/2} A_T - K_M^\ast R_M^{-1}K_M.
\end{align}
The stabilizing solution for the Riccati equation above exists since $A_T$ is stable by Lemma \ref{lemma:nabla}.

More importantly, note that $\Lambda$ has already a low rank $n$, but our objective is to obtain a matrix $\Lambda(M)$ whose rank is $m<n$. In the case where $B_w$ is square, the factorization is not needed since choosing $M=0$ already gives the desired rank.

To summarize, we have shown that
\begin{align}
    M(z)&= R_M^{1/2}\begin{pmatrix}
     K_M & I
    \end{pmatrix}
    \begin{pmatrix}
    \caus{A_T}B_w \\ I
    \end{pmatrix}\nn\\
    &= R_M^{1/2}(K_M \caus{A_T}B_w + I).
\end{align}
The inverse follows immediately as
\begin{align}
    M^{-1}(z)&= (I +  K_M \caus{A_T}B_w)^{-1}R_M^{-/2}\nn\\
    &= (I - K_M \caus{A_T - B_wK_M }B_w )R_M^{-/2}\nn\\
    &\triangleq (I - K_M \caus{A_M}B_w )R_M^{-/2}
\end{align}
with $A_M = A_T - B_wK_M$.
\end{proof}

\begin{proof}[\textbf{Proof of Lemma \ref{lemma:decomposition}}]
By Lemma \ref{lemma:Delta}, Lemma \ref{lemma:M_general} and \eqref{eq:FG}, we have the transfer functions
\begin{align}
\Delta^{-*}(z^{-*})&= (I + B_u^\ast PB_u)^{-/2}\nn\\
    & \ \ \cdot (I + B_u^\ast(z^{-1}I-A^\ast)^{-1}K_{lqr}^\ast)^{-1}R^{1/2}\nn\\
    F^*(z^{-*}) &= B_u^\ast (z^{-1} I - A^\ast)^{-1}Q^{1/2}\\
    G(z)&= Q^{1/2}\caus{A}B_w\nn\\
    M(z)&= R_M^{1/2}(K_M \caus{A_T}B_w + I).
\end{align}
Consider the product of the anti-causal transfer functions $\Delta^{-*}(z^{-*})F^*(z^{-*})$ (omitting constants on the sides)
\begin{align}\label{eq:proof_decomposition_causal_product}
 & (I + B_u^\ast(z^{-1}I-A^\ast)^{-1}K_{lqr}^\ast)^{-1} B_u^\ast (z^{-1} I - A^\ast)^{-1}\nn\\
 &= B_u^\ast (I + (z^{-1}I-A^\ast)^{-1}K_{lqr}^\ast B_u^\ast)^{-1}  (z^{-1} I - A^\ast)^{-1}\nn\\
 &= B_u^\ast \anti{A_K^\ast}.
\end{align}
so we have $\Delta^{-\ast}(z^{-\ast})F^\ast(z^{-\ast}) = (I + B_u^\ast PB_u)^{-/2} B_u^\ast \anti{A_K^\ast}Q^{1/2}$.
Using a Lyapunov equation, we use a standard decomposition to write
\begin{align}\label{eq:dec_mid}
    &\anti{A_K^\ast}Q\caus{A}\nn\\
    &= \anti{A_K^\ast}A_K^\ast P + PA \caus{A} + P.
\end{align}
By multiplying the strictly causal part of \eqref{eq:dec_mid} with the constants on both sides, we have the first strictly causal part of the product $\Delta^{-*}(z^{-*}) F^*(z^{-*}) G(z)M(z)$
\begin{align}
    C_1(z)&= - (I + B_u^\ast PB_u)^{-/2} B_u^\ast PA\caus{A} B_w M^{-1}(z).\nn
\end{align}
We proceed to combine the anti-causal part of \eqref{eq:dec_mid} with~$M^{-1}(z)$
\begin{align}\label{eq:rand}
    &z^{-1}\anti{A_K^\ast}P B_w (I + K_M \caus{A_T}B_w)^{-1}R_M^{-/2}\nn\\
    &= z^{-1}\anti{A_K^\ast}P (I - B_wK_M \caus{A_M} )B_wR_M^{-/2}\nn\\
    &= z^{-1}\anti{A_K^\ast}P B_wR_M^{-/2}\nn\\
    & - z^{-1}\anti{A_K^\ast}P B_wK_M \caus{A_M} B_wR_M^{-/2}\nn\\
    &= z^{-1}\anti{A_K^\ast}P B_wR_M^{-/2}\nn\\
    & - z^{-1}[\anti{A_K^\ast}A_K^\ast U + UA_M \caus{A_M}+U]\nn\\
    &\cdot B_wR_M^{-/2},
\end{align}
where $U$ solves the Sylvester equation $U = A_K^\ast U A_M + P B_wK_M$. By separating \eqref{eq:rand} into its anti-causal and strictly causal counterparts, we conclude that
\begin{align}
A(z)&= - z^{-1} (I + B_u^\ast PB_u)^{-/2} B_u^\ast \anti{A_K^\ast}\nn\\
&\cdot  (P-A_K^\ast U) B_wR_M^{-/2},
\end{align}
is the anti-causal part, and that the remaining strictly-causal transfer function is
\begin{align}
    C_2(z)&= (I + B_u^\ast PB_u)^{-/2} B_u^\ast U \caus{A_M} B_wR_M^{-/2}.\nn
\end{align}
\end{proof}

\begin{proof}[\textbf{Proof of Lemma \ref{lemma:our_nehari}}]
The proof follows by applying \cite[Th. $9$]{sabag2021regret} to the strictly anticausal transfer function $zA(z)$ in \eqref{eq:lemma_decomp_SC}.
\end{proof}

\begin{proof}[\textbf{Proof of Lemma \ref{lemma:M}}]
In the case where $B_w$ is square, we can simply choose $M(z) = z\nabla^{-1}(z)G(z)$, and show that its inverse is causal and bounded. The expression for $M(z)$ follows from \eqref{eq:zNablaG}. Recall that $B_w$ has a full-column rank and therefore is invertible. The inverse can then be computed as
\begin{align}
M^{-1}(z)&= B_w^{-1}(I + A_T\caus{A_T})^{-1}Q^{-/2}R_T^{1/2}\nn\\
&=B_w^{-1}z^{-1}(zI-A_T)Q^{-/2}R_T^{1/2} \nn\\
&= B_w^{-1}(I-z^{-1}A_T)Q^{-/2}R_T^{1/2}.
\end{align}
Note that the transfer function corresponds to a finite-length impulse response and therefore is bounded.
\end{proof}

\end{document}